\documentclass{amsart}

\usepackage{amsmath,amssymb,amscd} 
\usepackage{ascmac}
\usepackage[arrow,matrix]{xy}
\usepackage{enumerate}

\usepackage[dvipdfmx]{graphicx}

\usepackage[english]{babel}
\usepackage{comment,url}
\usepackage[normalem]{ulem}
\usepackage{accents}

\usepackage[usenames]{color}

\newcommand{\red}{} \newcommand{\blue}{} \newcommand{\green}{} \newcommand{\violet}{} \newcommand{\bred}{} \newcommand{\magenta}{} \newcommand{\cyan}{}

\definecolor{darkgreen}{rgb}{0,0.2,0.2}

\newtheorem{thm}{Theorem}[section]
\newtheorem{prop}[thm]{Proposition}

\newtheorem{lem}[thm]{Lemma}

\theoremstyle{definition}
\newtheorem{dfn}[thm]{Definition}
\newtheorem{exa}[thm]{Example}
\newtheorem{rem}[thm]{Remark}

\def\rank{\mathop{\mathrm{rank}}\nolimits}

\def\Im{\mathop{\mathrm{Im}}\nolimits}
\def\Ker{\mathop{\mathrm{Ker}}\nolimits}
\def\Coker{\mathop{\mathrm{Coker}}\nolimits}
\def\Hom{\mathop{\mathrm{Hom}}\nolimits}

\def\Spec{\mathop{\mathrm{Spec}}\nolimits}
\def\Gal{\mathop{\mathrm{Gal}}\nolimits}

\def\GL{\mathop{\mathrm{GL}}\nolimits}

\def\mod{\mathop{\mathrm{mod}}\nolimits}
\def\Cl{\mathop{\mathrm{Cl}}\nolimits}
\def\lk{\mathop{\mathrm{lk}}\nolimits}

\newcommand{\mf}[1]{{\mathfrak{#1}}}

\newcommand{\bb}[1]{{\mathbb{#1}}}
\newcommand{\mca}[1]{{\mathcal{#1}}}

\newcommand{\LR}{\longleftrightarrow}
\newcommand{\inj}{\hookrightarrow}
\newcommand{\surj}{\twoheadrightarrow}
\newcommand{\act}{\curvearrowright}
\newcommand{\congto}{\overset{\cong}{\to}}

\newcommand{\N}{\bb{N}}
\newcommand{\Z}{\bb{Z}}
\newcommand{\Zp}{\bb{Z}_{p}}
\newcommand{\Q}{\bb{Q}}
\newcommand{\Qp}{\bb{Q}_{p}}
\newcommand{\R}{\bb{R}}
\newcommand{\C}{\bb{C}}
\newcommand{\F}{\bb{F}}

\newcommand{\p}{\mf{p}}

\renewcommand{\O}{\mca{O}}

\renewcommand{\(}{\left\{}
\renewcommand{\)}{\right\}}
\newcommand{\del}{\partial}

\newcommand{\ol}{\overline}

\newcommand{\wt}[1]{{\widetilde{#1}}}
\newcommand{\wh}[1]{{\widehat{#1}}}

\begin{document}

\title[On the Iwasawa invariants for links]{On the Iwasawa invariants for links \\ and Kida's formula} 

\author{Jun UEKI} 

\maketitle 
\begin{abstract}
Analogues of Iwasawa invariants in the context of 3-dimensional topology have been studied by M.~Morishita and others. 
In this paper, following the dictionary of arithmetic topology, 
we formulate an analogue of Kida's formula on $\lambda$-invariants in a $p$-extension of $\Zp$-fields for 3-manifolds. 
The proof is given in a parallel manner to Iwasawa's second proof, with use of $p$-adic representations of a finite group. 
In the course of our arguments, 
we introduce the notion of a branched $\Zp$-cover as an inverse system of cyclic branched $p$-covers of 3-manifolds, 
generalize the Iwasawa type formula, and compute the Tate cohomology of 2-cycles explicitly. 
\end{abstract}



\footnote[0]{keywords: link, rational homology 3-sphere, branched cover, Galois theory, Iwasawa theory, arithmetic topology.}

\section{Introduction} 
The analogy between primes and knots is known (\cite{Morishita2012}). 
Among other things, there is close analogy between 
Iwasawa theory and Alexander--Fox theory, and 
topological analogues of Iwasawa's class number formula are studied (\cite{HMM2006}, \cite{KM2008}).  
Let $p$ denote a fixed prime number, and let 
$\Zp$ denote the ring of $p$-adic integers through out this paper.  
We call a field $K$ simply a \emph{$\Zp$-field} if it can be obtained as a $\Zp$-extension of a finite number field $k$.  
If the Iwasawa $\mu$-invariants vanish, 
extensions of $\Zp$-fields \blue{resemble} 
those of function fields, 
and satisfy Kida's formula (\cite{Kida1980}), 
which is an arithmetic analogue of the classical Riemann--Hurwitz formula. 
In this paper, we formulate their topological analogue, 
\blue{in a very parallel manner to } 
Iwasawa's second proof in \cite{Iwasawa1981}, 
\blue{using $p$-adic representation theory of finite groups 
and 
the Tate cohomologies.} 

Here are the contents of this paper. 
In Section 2, we review some basic analogies in branched Galois theories for finite degrees. 
An analogy between unit groups and 2-cycle groups was pointed out in 
our previous paper (\cite{Ueki1}). We calculate the Tate cohomology of 2-cycles  
explicitly. We also discuss analogue objects of $S$-ideal groups and others. 
In Section 3, we review Iwasawa's class number formula of $\Zp$-fields. 
In Section 4, as an analogue of $\Zp$-extension, we introduce the notion of 
a \emph{branched $\Zp$-cover} of rational homology 3-spheres ($\Q$HS$^3$), 
namely, an inverse system of branched cyclic $p$-covers of $\Q$HS$^3$'s, and 
generalize the Iwasawa type formula:

\setcounter{section}{4}
\setcounter{thm}{8} 
\begin{thm} \blue{Let $\wt{M}=\(M_n\)_n$ be a branched $\Zp$-cover consisting of $\Q$HS$^3$'s. Then there are some $\lambda, \mu, \nu\in \Z$ satisfying $\#H_1(M_n, \Zp)=p^{\lambda n+\mu p^n +\nu}$ for any $n\gg 0$.}  
\end{thm}

We also employ the structure theorem of $\Lambda$-modules, and prove a $p$-adic variant of Sakuma exact sequence. 
Then we deduce an alternative proof for the Iwasawa type formula, and also a proposition on the direct limit of homology groups. 
In Section 5, we further introduce an analogous notion of \blue{a Galois} extension of $\Zp$-fields, 
say \blue{an \emph{equivariant Galois morphism} (or a \emph{branched Galois cover})} of branched $\Zp$-covers. 
It is a compatible system of branched covers on each layer 
between two branched $\Zp$-covers. 
In addition, we define chains of branched $\Zp$-covers by 
direct limits with respect to the transfer maps, 
and compute the Tate cohomologies for them. 
%
In Section 6, we review 
Kida's formula for an extension of $\Zp$-fields. 
In Section 7, we establish an analogue of Kida's formula for branched $\Zp$-covers: 
\setcounter{section}{7} 
\setcounter{thm}{0}
\noindent 
\begin{thm}
\emph{Let $f:\wt{N\,}\!\to \wt{M}$ be \blue{an equivariant Galois morphism of degree \red{$p$-power} of }
branched $\Zp$-covers of $\Q$HS$^3$, 
let $\ol{S}$ be the branch link of $f_0:N\to M$,  
\blue{and let $\mca{S}$ denote the inverse limit of the preimages of $\ol{S}$ in $\wt{N\,}\!$.}  
If $\ol{S}$ is \blue{infinitely} inert in $\wt{M}$ and 
$\mu_{\wt{M}}=0$, then $\mu_{\wt{N\,}\!}=0$. 
\bred{If in addition any component of $\ol{S}$ is not inert in $f_0$, then} 
\blue{the Iwasawa $\lambda$-invariants and}  
the branch indices $e_w$ of components $w$ of \bred{$\mca{S}$} satisfy}
$$\lambda_{\wt{N\,}\!}-1=\deg(f)(\lambda_{\wt{M}}-1)+\sum_{w\subset \blue{\mca{S}}}(e_w-1).$$
\end{thm}
\setcounter{section}{1}


\section{Theories for finite degrees}
To begin with, we briefly review some basic analogies between primes and knots. 
In addition, 
for finite Galois extensions of number fields and finite branched Galois covers of 3-manifolds, 
we study 
analogies between unit groups and 2-cycle groups, $S$-ideals and ``$S$-chains'', and compute their Tate cohomologies 
in an explicit way. 

\subsection{M$^2$KR-dictionary}
The analogy between 3-dimensional topology and number theory was first pointed out by B.~Mazur in the mid 1960's (\cite{Mazur1963}), 
and has been studied systematically 
by M.~Kapranov (\cite{Kapranov1995}), A.~Reznikov (\cite{Reznikov1997}, \cite{Reznikov2000}) and M.~Morishita  (\cite{Morishita2010}, \cite{Morishita2012}). 
In their dictionary of analogies, for example, knots and 3-manifolds correspond to primes and number rings respectively. 
The study of these analogies is christened ``arithmetic topology'' now.
Here is \blue{a basic dictionary} 
we shall use in this paper. 
For a number field $k$, 
$\mca{O}_{k}$ denotes the ring of integers. 

\begin{center}
\begin{tabular}{|c||c|}
\hline 
Number theory&3-dimensional topology\\
\hline 
number ring $\Spec \O_k$ or& (oriented, connected, closed)\\ 
$\ol{\Spec \O_k}=\Spec\O_k\cup\({\rm infinite\ primes}\)$& 3-manifold $M$ \\
\hline 
prime ideal $\p: \Spec \F_\p \inj \Spec \O_k$ & knot $K:S^1 \inj M$ \\
prime ideals $S=\{ \p_1,...,\p_r \}$
 & link $L=\{K_1,...,K_r\} $ \\ 
\hline 
extension $F/k$ & branched cover $f:N\to M$\\ 
\hline 
\'{e}tale fundamental group $\pi_{1}(\Spec \O_k$) & fundamental group $\pi_{1}(M)$\\
$\pi_{1}^{\text{\'et}}(\Spec\O_{k}-S$) & link group $\pi_{1}(M-L)$\\
\hline 
geometric point $\Spec \C \inj \Spec \mca{O}_k$ & base point $\(\rm pt\)  \inj M$\\ 
\hline 
ideal group $I(k)$ & 1-cycle group $Z_{1}(M)$\\
$k^{*} \to I(k);\ a \mapsto (a)$  & $C_{2}(M)\to Z_{1}(M);\ c\mapsto \partial c$ \\ 
principal ideal group $P(k)$& 1-boundary group $B_{1}(M)$\\
\hline 
\end{tabular}
\ \\ 
\begin{tabular}{|c||c|}
\hline 
ideal class group $\Cl(k)=I(k)/P(k)$& 1st-homology $H_{1}(M)=Z_{1}(M)/B_{1}(M)$\\
Fact: $\#\Cl(k)<\infty$ & \blue{Condition:} $\#H_1(M)<\infty$ (i.e. $M$: $\Q$HS$^3$),\\
 & or consider torsion subgroup $H_1(M)_{\rm tor}$\\
\hline
Artin reciprocity & Hurewicz isomorphism\\
$\pi_1^{\text{\'et}}(\ol{\Spec\mca{O}_k})^{\rm ab}\cong \Gal(k^{\rm ur}_{\rm ab}/k) \cong \Cl(k)  $ 
& $\pi_1(M)^{\rm ab}\cong \Gal(M_{\rm ab}/M) \cong H_{1}(M)$\\
$k^{\rm ur}_{\rm ab}/k$ Hilbert class field 
& $M_{\rm ab}\to M$ maximal abelian cover\\ 
\hline
unit group $\O_{k}^{*}$ & 2-cycle group $Z_{2}(M)$,\\
& or 2nd-homology $H_{2}(M)\ (\cong H_1(M)_{\rm free})$\\
\hline
\end{tabular}
\end{center} 

We assume that number fields are finite over $\Q$ and are contained in $\C$, 
3-manifolds are oriented, connected, and closed. 
A branched cover means an isomorphism class of branched covers \blue{with base points} branched over links. 

We have two attitudes about analogues of ideal class groups $\Cl(k)$ and unit groups $\mca{O}_k^*$. 
In the conventional attitude, we do not assume $M$ to be $\Q$HS$^3$, and 
we 
\blue{regard} the torsion subgroups $H_1(M)_{\rm tor}$ as an analogue of $\Cl(k)$. 
Then a non trivial term $H_2(M)\cong H_1(M)_{\rm free}$ plays an analogous role to $\mca{O}_k^*$ (\cite{Sikora2003}, \cite{Morin2008}).  
In another one, which we pointed out in \cite{Ueki1}, we assume $M$ to be a $\Q$HS$^3$, fix a CW-structure, and \blue{regard} 
$Z_2(M)$ as an analogue of $\mca{O}_k^*$. Then, we have the following parallel exact sequences: 

\begin{center}
\begin{tabular}{|c||c|}
\hline
$1\to P(k)\to I(k)\to \Cl(k)\to 1$ & $0\to B_{1}(M)\to C_{1}(M)\to H_{1}(M)\to 0$\\
\hline
$1\to \O_{k}^{*}\to k^{*}\to P(k)\to 1$ & $0\to Z_{2}(M)\to C_{2}(M)\to B_{1}(M)\to 0$\\
\hline
\end{tabular}
\end{center} 

In addition, we have the following theorem. 
\begin{thm}[Ueki \cite{Ueki1}] 
Let $f:N\to M$ be a branched Galois cover of 3-manifolds. 
Let $G=\Gal(h)$, and fix CW-structures compatible with $h$. 
Then the Tate cohomology of $\wh{H}^i(G,Z_2(N))$ 
is independent of the choice of $CW$-structures, 
and is a topological invariant of branched covers. 
\end{thm} 
\blue{We also remark that $\wh{H}^i(G,Z_2(N))$ has more information than $\wh{H}^i(G,H_2(N))$. 
These above enable further translation. 
In our previous paper \cite{Ueki1}, we gave an analogue of Yokoi's formulation of genus theory (\cite{Yokoi1967}).}

The Tate cohomology of the unit group played an important role in Iwasawa's second proof of Kida's formula\blue{. In this article,} 
that of the 2-cycle group 
\blue{plays} a similar role in the proof of our main result. 
\subsection{Computations of Tate cohomologies}  
In this subsection, we \blue{present} some explicit computations of Tate cohomologies \blue{for number fields extensions and branched covers of 3-manifolds.} 
For a group $G$ and a $G$-module $A$, $\wh{H}^i(G,A)$ denotes the Tate cohomology \magenta{for each $i\in \Z$}. We abbreviate $\wh{H}^i(G,A)=\wh{H}^i(A)$ \blue{in the proofs} if there is no ambiguity of $G$. 

\blue{The following facts in number theory are well-known.} 
\begin{prop} Let $F/k$ be a Galois extension of number fields \blue{with $G=\Gal(F/k)$.} 
Then,\\ 
(1) {\rm [Hilbert's Satz 90]} \blue{T}he equality $\wh{H}^{\magenta{1}}(G,F^*)=0$ holds.\\
(2) {\rm (Iwasawa \cite{Iwasawa1956u})} If $F/k$ is unramified and $\Cl(F)=1$, then 
$\wh{H}^i(G,\mca{O}_F^*)\congto \wh{H}^{i-1}(G,C_F)\congto \wh{H}^{i-3}(G,\Z)$ \magenta{for each $i \in \Z$}, 
\blue{where $C_F$ denote the idele class group of $F$.}   
Some cases with restricted ramifications, such as cyclotomic extensions,  are also computable. 
\end{prop}

We have the following on the topology side. 
\begin{prop}
Let $f: N\to M$ be a branched Galois cover of 3-manifolds, 
Put $G=\Gal(f)$, and fix a CW-structures or PL-structures on $M$ and $N$ compatible with $f$. Then \magenta{for each $i\in \Z$},\\ 
(1) 
The equality $\wh{H}^i(G,C_2(N))=0$ holds. \\
(2) There is \blue{a} long exact sequence 
$\cdots \to \wh{H}^{i+1}(G,H_3(N))\to \wh{H}^i(G,Z_2(N))\to \wh{H}^i(G,H_2(N))\to \wh{H}^{i+2}(G,H_3(N))\to \cdots .$ 
\blue{If in addition} $N$ is a $\Q$HS$^3$\blue{,} 
then $\wh{H}^i(G,Z_2(N))\cong \wh{H}^{i+1}(G,H_3(N))$\blue{, where}
$H_3(N)=\langle[N]\rangle\cong \Z$. 
\end{prop}
\begin{proof} 
(1) Since the subset of $G$-fixed points is 1-dimensional, $C_2(N)$ is $\Z[G]$-free. 
(2) \blue{The} exact sequence $0\to B_2(N)\to Z_2(N)\to H_2(N)\to 0$ induces a long exact sequence  
$\cdots \to \wh{H}^i(B_2(N))\to \wh{H}^i(Z_2(N))\to \wh{H}^i(H_2(N))\to \wh{H}^{i+1}(B_2(N))\to \cdots .$ We 
\blue{prove} $\wh{H}^i(B_2(N))\cong \wh{H}^{i+1}(H_3(N))$. 
By the same reason as (1), $C_3(N)$ is a $\Z[G]$-free module, 
and $\wh{H}^i(C_3(N))=0$. Hence 
\blue{the} exact sequence $0\to Z_3(N)\to C_3(N)\to B_2(N)\to 0$ yields 
an isomorphism $\wh{H}^i(B_2(N))\cong \wh{H}^{i+1}(Z_3(N))$. 
In addition, since $N$ is 3-dimensional, we have $B_3(N)=0$. 
Hence \blue{the} exact sequence $0\to B_3(N)\to Z_3(N)\to H_3(N)\to 0$
yields isomorphisms $Z_3(N)\cong H_3(N)=\langle[N]\rangle\cong \Z$. 
Thereby, we obtain $\wh{H}^i(B_2(N))\cong \wh{H}^{i+1}(Z_3(N))\cong \wh{H}^{i+1}(H_3(N))$. 
Especially, if $N$ is a $\Q$HS$^3$, the Poincar\`e duality and the universal coefficient theorem ensure $H_2(N)=0$, and hence 
$\wh{H}^{i+1}(H_3(N))\congto \wh{H}^i(Z_2(N))$. 
\end{proof} 
\subsection{$S$-ideals and $S$-chains} 
In this subsection, we \blue{recall some properties of $S$-ideals and others of number fields \magenta{and $\Zp$-fields}. 
Then we introduce their analogues for 3-manifolds,} 
and obtain further results of Tate cohomologies.  

Let $F/k$ be a finite Galois extension of number fields \magenta{or $\Zp$-fields}. 
\magenta{Let} $S$ be a finite set of primes in \violet{$k$, and \magenta{let $S$} also \magenta{denote} the set of primes above $S$ in $F$}. 
Then \magenta{the $S$-ideal group $I_{F,S}$, the principal $S$-ideal group $P_{F,S}$, the $S$-ideal class group $\Cl_{F,S}$, and the $S$-unit group $\mca{O}^*_{F,S}$} 
are defined, and there are exact sequences of $G=\Gal(F/k)$-modules (\cite{Brumer1966}, \cite{Iwasawa1981}, \cite{NSW}). 

\begin{center}
$0\to P_{F,S}\to I_{F,S}\to \Cl_{F,S}\to 0$\\
$0\to \mca{O}^*_{F,S}\to F^* \to P_{F,S}\to 0$
\end{center}

For a prime $\p$ in $k$, let $I_{F,\p}$ denote the group of ideals of $F$ over $\p$, and let $Z_{\p}$ denote the decomposition group of $\p$. Then we have isomorphisms 
$\wh{H}^i(G,I_{\p})\cong \wh{H}^i(G,\Z[G/Z_{\p}])\cong \wh{H}^i(Z_{\p},\Z)$. 
The direct sum decomposition $I_F=\oplus_{\p} I_{F,\p}$ yields 
$\wh{H}^i(G,I_F)=\oplus_{\p} \wh{H}^i(G,I_{F,\p})\cong \oplus_{\p} \wh{H}^i(Z_{\p},\Z)$. 

\begin{prop}[Iwasawa \magenta{\cite{Iwasawa1981}}] 
Let $F/k$ be an extension of degree $p$ unramified at infinite primes, \blue{put $G=\Gal(F/k)$,} and let $S$ \blue{denote} 
the set of ramified non-$p$ primes. Then, $S$ is a finite set. \\
(1) The Tate cohomology $\wh{H}^i(G,I_{F,S})=\oplus_{\p \not\in S}\wh{H}^i(G,I_{\p})$ \blue{of $S$-ideals} \magenta{($i \in \Z)$} vanishes if and only if no ideals outside $S$ inert. (If $k$ is a number field, then it does not vanish.)\\ 
\blue{(2) \magenta{Let $I_S$ denote the subgroup of $I(F)$ generated by elements of $S$. Then the}
$S$-unit group satisfies $\mca{O}^*_{F,S}/\mca{O}^*_F\cong \magenta{P(F)}\cap I_S$ for a subgroup $\magenta{P(F)}\cap I_S<I_S\cong \Z^{\#S}$ of finite index, and 
$\wh{H}^1(G, \mca{O}^*_{F,S}/\mca{O}^*_F)=0, \wh{H}^2(G, \mca{O}^*_{F,S}/\mca{O}^*_F) \cong (\Z/p\Z)^{\#S}$.} 
\end{prop} 

We consider analogues of these above. 
Let $f:N \to M$ be a finite Galois branched cover of 3-manifolds,
put $G=\Gal(f)$ and 
fix CW-structures or PL-structures compatible with $f$. 
Let $\ol{S}\subset M$ be a link \violet{and suppose that its each component is in or} \blue{outside the branch link of $f$}. 
Note that analogues of $S$-ideals are \emph{not} the ones obtained from $C_*(N,S):=\Coker(C_*(S)\inj C_*(N))$. 
Instead, we consider the following commutative diagram, whose rows are exact. 
$$
\xymatrix{
0 \ar[r] & B_1(S)=0 \ar[r] \ar[d] & Z_1(S) \ar^{\cong}[r] \ar^{\iota_*}[d] & H_1(S) \ar[r] \ar^{\iota_*}[d] & 0\\
0 \ar[r] & B_1(N) \ar[r] & Z_1(N) \ar[r] & H_1(N) \ar[r] & 0
} $$
\magenta{We identify $Z_1(S)$ and $\iota_*(Z_1(S))$.}
We 
put $Z_1(N)_S:=Z_1(N)/\magenta{Z_1(S)}$, $H_1(N)_S:=H_1(N)/\iota_*(H_1(S))$, 
and $B_1(N)_S:= B_1(N)/(B_1(N)\cap Z_1(S))$. 
\red{In addition,} for the boundary map $\del:C_2(N)\to C_1(N)$, 
we put $Z_2(N)_S:=\del^{-1}(Z_1(S))$. \blue{Then we have $Z_2(N)_S\cong Z_2(N,S)$.}
\blue{Thus} we obtain exact sequences of $G$-modules parallel to the case of number theory. 
\begin{center}
$0\to B_1(N)_S \to Z_1(N)_S\to H_1(N)_S\to 0$\\
$0\to Z_2(N)_S \to C_2(N)\to B_1(N)_S\to 0$
\end{center} 

\blue{The} analogue of the $S$-ideal group satisfies the following proposition. 

\begin{prop} \blue{Let $f:N\to M$ be a finite Galois branched cover with $G=\Gal(f)$, let $\ol{S}\subset M$ be a link, 
and put $S=f^{-1}(\ol{S})$.} \\ 
{\rm (i)} If $S$ is not empty, then there is an isomorphism $\wh{H}^i(G,Z_1(N)_S)\cong \wh{H}^{i-1}(\blue{G,}\Ker(\iota_*:H_0(S)\to H_0(N)))$ \magenta{for each $i\in \Z$}. 
\blue{If in addition} $f$ is of degree $p$ and \blue{$\ol{S}$ is the branch link with $s$ components,} 
then 
$\wh{H}^0(G,Z_1(N)_S)=0$ and $\wh{H}^1(G,Z_1(N)_S)\cong  (\Z/p\Z)^{s-1}$.\\ 
{\rm (ii)} If $S$ is empty so that $f$ is unbranched, then  
$\wh{H}^i(G,Z_1(N))\cong \wh{H}^i(\blue{G},H_0(N)))\cong \wh{H}^i(\blue{G},\Z)$ \magenta{for each $i\in \Z$}. 
If \blue{in addition} $\deg(f)=p$,  
then $\wh{H}^0(G,Z_1(N))\cong \Z/p\Z,$ $\wh{H}^1(G,Z_1(N))=0$. 
\end{prop}

\begin{proof} If $S$ is not empty, then 
\blue{the} exact sequences 
$0\to Z_1(S)\to C_1(S)\to B_0(S)\to 0$,  
$0\to Z_1(N)\to C_1(N)\to B_0(N)\to 0$ 
and the snake lemma yield an exact sequence 
$0\to Z_1(N)_S\to C_1(N)/C_1(S) \to B_0(N)/B_0(S)\to 0$. 
Here, $C_1(N)/C_1(S)$ is a $\Z[G]$-free module. %
In addition, \blue{the} exact sequences 
$0\to B_0(S)\to Z_0(S)\to H_0(S)\to 0$, 
$0\to B_0(N)\to Z_0(N)\to H_0(N)\to 0$ 
and the snake lemma yield an exact sequence 
$0\to \Ker(\iota_*:H_0(S)\to H_0(N))\to B_0(N)/B_0(S)\to  Z_0(N)/Z_0(S) \to 0$. 
Here, $Z_0(N)/Z_0(S)$ is a $\Z[G]$-free module. 
As a consequence, isomorphisms 
 $\wh{H}^i(Z_1(N)_S)\cong \wh{H}^{i-1}(B_0(N)/B_0(S))
\cong \wh{H}^{i-1}(\Ker(\iota_*:H_0(S)\to H_0(N))$ are obtained. 
Especially, if $\deg(f)=p$ and $\ol{S}$ is properly branched, then 
$\Ker(\iota_*:H_0(S)\to H_0(N))\cong \Z^{s-1}$. 

If $S$ is empty, we consider \blue{the} exact sequences  
$0\to Z_1(N)\to C_1(N)\to B_0(N)\to 0$ and 
$0\to B_0(N)\to Z_0(N)\to H_0(N)\to 0$, 
in which $C_1(N)$ and $Z_0(N)$ are $\Z[G]$-free modules, and $H_0(N)\cong \Z$. Then, isomorphisms 
$\wh{H}^i(Z_1(N)_S)=\wh{H}^i(Z_1(N))\cong \wh{H}^{i-1}(B_0(N))\cong \wh{H}^{i-2}(H_0(N))\cong \wh{H}^i(\Z)$ are obtained. \end{proof}

If we fix a $\Z$-basis of $Z_1(M)$ and \blue{regard} 
it as an analogue of the set of primes in the base field, then we obtain a direct sum decomposition of $Z_1(N)$ as $G$-modules. 

\blue{The analogue of the $S$-unit group satisfies the following.}  
\begin{prop}
\blue{Let $f:N\to M$ be a Galois branched cover of degree $p$ with $G=\Gal(f)$, 
let $\ol{S}$ be the branch link with $s$-components, 
and put $S=f^{-1}(\ol{S})$. Suppose that $S$ consists of $\Q$-null-homologous components. (This assumption holds if $N$ is a $\Q$HS$^3$.) Then $S$-2-cycles satisfy $Z_2(N)_S/Z_2(N)\cong B_1(N)\cap Z_1(S)$ for a subgroup $B_1(N)\cap Z_1(S)<Z_1(S)\cong \Z^s$ of finite index, and 
$\wh{H}^1(G,Z_2(N)_S/Z_2(N))=0,$ $\wh{H}^2(G,Z_2(N)_S/Z_2(N))\cong (\Z/p\Z)^s$.} 
\end{prop}

\begin{proof} 
\blue{Note that $Z_2(N)_S\cong Z_2(N,S)$. 
There is a natural exact sequence $0\to Z_2(N)\to Z_2(N,S)\overset{\del}{\to} Z_1(S)$. 
Since every component of $S$ is $\Q$-null-homologous, $\del Z_2(N,S)=B_1(N)\cap Z_1(S)< Z_1(S)$ is a subgroup of finite index. Since $S$ is the preimage of the branch link, $G$ acts on $B_1(S)$ trivially.}   
\end{proof}

\blue{These computations above will be extended for $\Zp$-covers in \magenta{Subsection} 5.4.} 


\section{Iwasawa's class number formula of $\Zp$-extensions and $\Lambda$-modules}
It has been known that 
there is close analogy between Iwasawa theory and Alexander--Fox theory  (\cite{Mazur1963}, \cite{Morishita2012}). 
For an inverse system of branched cyclic $p$-covers of $\Q$HS$^3$ obtained from a $\Z$-cover, 
an analogue of Iwasawa's class number formula is formulated  (\cite{HMM2006}, \cite{KM2008}, \cite{KM2013}). 
In this paper, 
we generalize this formula for an inverse system 
which is not necessarily obtained from a $\Z$-cover. 

\begin{center}
\begin{tabular}{|c||c|}
\hline Iwasawa theory& Alexander--Fox theory\\
\hline
\hline
$\Zp$-extension of $k$ & $\Z$-cover over $M - L$, or\\
&branched $\Zp$-cover over $(M,L)$\\
\hline
Iwasawa module & link module \\
Iwasawa polynomial & Alexander polynomial\\
\hline
\end{tabular}
\end{center}
In this section, we recall algebraic lemmas on $\Lambda$-modules,  
a proof of Iwasawa's class number formula, 
and an assertion on some direct limit module. 
Their analogues will be discussed in the next section. 

Let $\Lambda=\Zp[[T]]$ denote the ring of formal power series. 
For $\Lambda$-modules $M$ and $M'$, a \emph{pseudo isomorphism} $M\sim M'$ is a homomorphism with finite kernel and cokernel. 
For finitely generated compact $\Lambda$-modules, pseudo isomorphisms give an equivalence relation. 
There is a structure theorem on compact $\Lambda$-modules. 

\begin{lem}[Whasington \cite{Washington}, Chapter 13]
\label{Lambda} Let $E$ be a compact $\Lambda$-module. \\
(1) {\rm (Nakayama's lemma)} $E$ is a finitely generated $\Lambda$-module if and only if $E/(p,T)$ is a finite group.\\
(2) Let $E$ be a finitely generated $\Lambda$-module. Then, there is a pseudo isomorphism 
$$E\sim \Lambda^{\oplus r}\oplus (\oplus_i \Lambda/(f_i^{e_i}))\oplus (\oplus_j\Lambda/(p^{m_j}))$$
to a unique normal form. 
Here, $r,e_i, m_j \in \N$, and $f_i\in \Zp[T]$ is a irreducible Weierstrass polynomial, namely, the coefficient of its highest term is 1 and others are multiples of $p$. \\ 
(3) Let $\mu=\sum m_j, \lambda=\sum e_i \deg (f_i)$ in the normal form above, and put $\nu_{p^n}:=(t^{p^n}-1)/(t-1)=\sum_{0\leq i<p^n}t^i$. 
If $E/\nu_{p^n}E$ is finite for all $n$, then $r=0$, and there is some $\nu,n_0$ such that for any $n>n_0$,  
$$E/\nu_{p^n}E=p^{\lambda n+ \mu p^n +\nu}.$$
(4) In this situation, $\mu=0$ if and only if the $p$-rank of $E/\nu_{p^n}E$ is bounded. 
\end{lem} 

Next, we review number theory. 
A \emph{$\Zp$-field}\footnote{A \emph{cyclotomic $\Zp$-field} and a cyclotomic $\Zp$-field which includes $p$-power-th roots of unity in this paper are a \emph{$\Zp$-field} and a \emph{cyclotomic $\Zp$-field} in the sense of Iwasawa in \cite{Iwasawa1981} respectively.} is a field obtained as a $\Zp$-extension of a \magenta{number field}. 
\magenta{If $k_\infty/k$ is a $\Zp$-extension of a number field, then $k_\infty$ is the direct limit (union) of $\Z/p^n\Z$-extensions $k_n/k$.} 
We assume that $\Zp$-fields are also contained in $\C$. A famous formula is stated as follows: 

\begin{thm}[Iwasawa's class number formula \cite{Iwasawa1959}] \label{ICF}
Let $\blue{k_\infty}/k$ be a $\Zp$-\blue{extension} of a finite number field. 
For each \magenta{$n\in \N=\N\cup \{0\}$}, 
let $k_n/k$ denote the subextension of degree $p^n$, and 
let $\Cl(k_n)_{[p]}$ 
\blue{denote} the 
\blue{$p$-}ideal class group.
Then there are some $\lambda,\mu \in \N$, 
$\nu \in \Z$, and $n_0\blue{\in \N}$ such that for any $n>n_0$, 
$$\#\Cl(k_n)_{[p]}=p^{\lambda n+\mu p^n +\nu}.$$ 
\end{thm}
These $\lambda, \mu,$ \magenta{and} $\nu$ are called the \emph{Iwasawa invariants}, 
\magenta{and denoted as $\lambda_{k_\infty/k}, \mu_{k_\infty/k}$, and $\nu_{k_\infty/k}$.
The value of $\lambda$ and whether $\mu=0$ or not depend only on the $\Zp$-field $k_\infty$, and are independent of the base field $k$. Hence they are sometimes expressed as $\lambda_{k_\infty}$ and $\mu_{k_\infty}=0$.}
Let $\Q_\infty$ denote the unique $\Zp$-extension over $\Q$. 
For a number field $k$, \blue{\emph{the cyclotomic $\Zp$-extension} of $k$ is defined \magenta{by} $k_\infty^{\rm cyc}:=k\Q_\infty$, and the Iwasawa invariants of $k$ are defined by those of $k_\infty^{\rm cyc}/k$.} 

\blue{
In the following, we review a proof of this formula with use of Lemma 3.1 on $\Lambda$-modules. 
Let $k_\infty/k$ and $k_n$ be as above, let $t$ be a topological generator of  $\Gamma:=\Gal(k_\infty/k)\cong \Zp$, and fix an identification $\Lambda=\Z[[T]]\cong \Zp[[\Gamma]]; 1+T \LR t$. The group $\Gamma$ acts on the inverse system $\(\Cl(k_n)_{[p]}\)_n$ and hence \emph{the Iwasawa module} $\mca{H}:=\varprojlim \Cl(k_n)_{[p]}$ continuously, and this action makes $\mca{H}$ a compact $\Lambda$-module. The $p$-class groups $\Cl(k_n)_{[p]}$ can be expressed as a quotient of $\mca{H}$:} 


\begin{prop}[Washington \cite{Washington}, Proposition 13.22 and Chapter 13.3] \label{classnumber}
Let $\blue{k_\infty}/k$ be a $\Zp$-extension with $n$-th subfield $k_n$, and put $\mca{H}\blue{=}\varprojlim \Cl(k_n)_{[p]}$.\\
(1) When precisely one prime is ramified in $k_\infty/k$ and it is totally ramified, 
for each $n$, (i) there is an isomorphism $$\Cl(k_n)_{[p]}\cong \mca{H}/(t^{p^n}-1)\mca{H},$$ 
and (ii) $p|\#\Cl(k)$ if and only if $p|\#\Cl(k_n)$. \\
(2) More generally, if every ramified prime in $k_\infty/k$ is totally ramified, 
then there is a subgroup $\mca{H}'<\mca{H}$ of finite index such that 
$\Cl(k_n)_{[p]}\cong \mca{H}/\nu_{p^n}\mca{H}'$. 
\end{prop} 
The number of ramified primes in a $\Zp$-extension is always finite.  General cases reduce to the cases of totally ramified. Indeed, 
for any $\Zp$-extension $k_\infty/k$, there is some $n_0$ such that 
$k_\infty/\red{k_{n_0}}$ is totally ramified at every ramified prime. 

The following lemma can be obtained from Lemma \ref{Lambda} (1): \begin{lem}[Washington \cite{Washington}, Chapter 13.3]\label{classFT}
The \blue{Iwasawa module} $\mca{H}$ is a finitely generated torsion $\Lambda$-module. 
\end{lem} 

Now Theorem \ref{ICF} (Iwasawa's class number formula) follows immediately from Lemma \ref{classFT}, Proposition \ref{classnumber}, and Lemma \ref{Lambda} (3) (The structure theorem of compact $\Lambda$-modules). \\

On the other hand, in a $\Zp$-extension $k_\infty/k$, 
ideal class groups of $k_n$ form an inductive system with respect to the maps induced by the natural injections of ideal groups $I_{k_n}\inj I_{k_{n+1}}$. 
The ideal class group of $k_\infty=\cup k_n$ is defined by $\Cl(k_\infty):=I(k_\infty)/P(k_\infty)$, 
where $I(k_\infty)=\varinjlim I(k_n), P(k_\infty)=\varinjlim P(k_n)$, and satisifies $\Cl(k_\infty)=\varinjlim \Cl(k_n)$. 
We have the following proposition. 

\begin{prop}[Iwasawa \cite{Iwasawa1981}, a remark in Chapter 5] 
Let $\blue{k_\infty}/k$ 
be a cyclotomic $\Zp$-extension. Then, there is an isomorphism of discrete $\Lambda$-modules 
$$\Cl(k_\infty)=\varinjlim \Cl(k_n)_{[p]} \cong (\Qp/\Zp)^{\lambda_{k_\infty/k}} \oplus A',$$
where $A'$ is a bounded module, namely, there is some $a\in \N$ such that $p^aA'=0$, and $\mu_{k_\infty/k}=0$ if and only if $A'=0$.
\end{prop}
Such a group is discussed also in Iwasawa \cite{Iwasawa1973Zl}, Chapter 5.

\section{Iwasawa type formula of branched $\Zp$-covers and $\Lambda$-modules}

In this section, we first introduce the notion of a branched $\Zp$-cover of 3-manifolds as an analogue of a $\Zp$-extension, and explain that it essentially generalizes the conventional objects.   
Next, we prove the Iwasawa type formula \blue{on the homology growth} for a branched $\Zp$-covers of $\Q$HS$^3$\blue{'s}. 
Moreover, we state a $p$-adic variant of Sakuma's exact sequence, whose proof will be given afterwards, 
and deduce an alternative proof of the Iwasawa type formula in a parallel manner to the case of number theory. 
In addition, we deduce a proposition on the direct limit of homology groups which will be used in Section 7. 
We also discuss further generalization to the cases of non-$\Q$HS$^3$ briefly.
\subsection{Branched $\Zp$-covers}
In this subsection, we discuss an analogue of $\Zp$-extension. 
\begin{dfn}
Let $M$ be a 3-manifold and let $L$ be a link in $M$. 
A \emph{branched $\Zp$-cover} over $(M,L)$ is 
an inverse system $\wt{M}=\(h_n:M_n\to M\)_n$ of cyclic branched covers of $M$ branched over $L$ with degree $p^n$. 
It is a branched $\Zp$-cover \emph{of $\Q$HS$^3$} if all $M_n$ are $\Q$HS$^3$. \end{dfn}

\begin{prop}
Let $L$ be a link in a 3-manifold $M$ and put $X=M-L$. 
A branched $\Zp$-cover over $(M,L)$ corresponds to 
a homomorphism $\tau:\pi_1(X)\to \Zp$ such that $\tau \mod p \neq 0$ uniquely up to ${\rm Aut}(\Zp)$. 
\end{prop}


\begin{proof} 
If such $\tau$ is given, then 
the composite 
$\tau_n:=(\Zp\surj \Z/p^n\Z)\circ \tau$ 
with the natural surjection 
is a surjective homomorphism for each $n\in \N$. 
A subgroup $\Ker(\tau_n)<\pi_1(X)$ corresponds to a cyclic cover $h_n:X_n\to X$ with degree $p^n$, and 
hence a cyclic branched cover $h_n:M_n\to M$
by Fox completion. 
The family $\(h_n\)_n$ forms an inverse system in a natural way. 

Conversely, if such an inverse system $\wt{M}=\(h_n:M_n\to M\)_n$ is given, 
a family of surjective homomorphism $\tau_n:\pi_1(X)\surj \Z/p^n\Z$ is obtained. 
Since $h_n$ is a subcover of $h_{n+1}$ for each $n$, $\Ker \tau_n<\Ker\tau_{n+1}$ holds, and there is a surjective homomorphism $q_n:\Z/p^{n+1}\Z\surj \Z/p^n\Z$ which satisfies $q_n\circ \tau_{n+1}=\tau_n$. 
The inverse limit of an inverse system \blue{$\(\Z/p^n\Z, q_n\)_n$} is isomorphic to  $\Zp$, and hence $\tau:\pi_1(X)\to \Zp$ is obtained up to ${\rm Aut}(\Zp)$. 
(In order to obtain $\tau$ explicitly, for each $n$, we replace $\tau_n$ by the composite of $\tau_n$ and an element of ${\rm Aut}(\Z/p^n\Z)$ such that $q_n$ is the natural surjection $\Zp\to \Z/p^n\Z$.)  
Such $\tau$ is unique up to ${\rm Aut}(\Zp)$. \end{proof}

The assumption $\tau \mod p \neq 0$ is equivalent to 
that $\tau$ sends a generator of $\pi_1(X)$ to a unit of $\Zp$, 
that $\tau$ has a dense image, 
and that $\tau$ induce a surjection $\tau: \blue{\wh{\pi}}_1(X) \surj \Zp$ from the pro-$p$ completion \blue{defined as} $\blue{\wh{\pi}}_1(X):=\varprojlim_{N} \pi_1(X)/N$ \blue{with} $N$ \blue{running} through the set of normal subgroups with $p$-power indices. 

\begin{rem} \blue{Let $\wt{M}$ be a branched $\Zp$-cover as above. 
The Fox completion is defined for more general objects than manifolds called spreads (\cite{Fox1957}), and $M_\infty:=\varprojlim M_n$ is the Fox completion of $X_\infty:=\varprojlim X_n$, which is not necessarily a manifold. 
There are some researches on \magenta{such objects} (\cite{Dellomo1986}, \cite{Dellomo1988}, \cite{CS1977}).} 
\end{rem}


Since the image of $\tau:\pi_1(X)\to \Zp$ is abelian, $\tau$ factors through $\pi_1(X)^{\rm ab}\cong H_1(X)$. 
In order to see examples, we 
\blue{prepare} the following lemma \blue{on $H_1(X)$}. 
It was originally given by \cite{KM2008} Lemma 4.2 for $\Q$HS$^3$, but 
it can be prove for a general (oriented, connected, and closed) $M$ in a similar way with use of intersection form. 
\begin{lem} 
Let $L=\sqcup K_i$ be a $d$-component link 
in a 3-manifold $M$, put $X=M-L$, and let $\mu_i\in H_1(X)$ denote the meridian of $K_i$ for each $i$. 
Then that $L$ consists of null-homologous components 
is equivalent to that \blue{the natural exact sequence} 
$$0\to \langle \mu_1,...,\mu_d\rangle\to H_1(X)\to H_1(M)\to 0$$
with $\langle \mu_1,...,\mu_d\rangle\cong \Z^d$ \blue{splits},  
and that $H_1(X)_{\rm free}:=H_1(X)/H_1(X)_{\rm tor}$ has a $\Z$-basis containing the image of $(\mu_1,\ldots,\mu_d)$. 
\end{lem}
\begin{proof} 
Fix a tubular neighborhood $V_L=\sqcup V_{K_i}$ of $L$ 
and put $X^{\circ}=M-{\rm Int}(V_L)$. Then there is a non-degenerate quadratic form $I:H_2(X^\circ,\del X^\circ)_{\rm free}\times H_1(X^\circ)_{\rm free}\to \Z$ called \emph{the intersection form}. 
Let $\mu_i$ denote the meridians of $K_i$ in \red{all of} $H_1(\del V_{K_i})$, $H_1(X^\circ)\cong H_1(X)$, and $H_1(X)_{\rm free}$ for each $i$. 

Suppose that $[K_i]=0$ in $H_1(M)$ for all $i$. Then for each $i$, there is a surface $S'_i$ in $M$ such that  $\del S'_i=K_i$. We may assume that ${\rm Int}(S'_i)$ intersects both $L$ and $\del V_L$ transversely. Put $S_i:=S_i'\cap X^{\circ}$. 
Since there is a basis of $H_1(\del X^{\circ})$ containing $(\del [S_1],\cdots, \del [S_d])$,  
there is a $\Z$-basis of $H_2(X^{\circ},\del X^{\circ})_{\rm free}$ containing $([S_1],\cdots, [S_d])$. 
Replace elements $T_j$ of this basis other than $S_i$'s so that $I(T_j,\mu_i)=0$. 
Then the dual basis of $H_1(X)$ with respect to $I$ contains $(\mu_1,\cdots,\mu_d)$, 
and its lift defines an isomorphism $H_1(X)\congto \langle \mu_1,\cdots,\mu_d\rangle_{\Z} \oplus H_1(M)$. 


Conversely, suppose that there is a $\Z$-basis of $H_1(X)_{\rm free}$ containing $(\mu_1,\ldots,\mu_d)$,  
and take the dual basis of $H_2(X^{\circ}, \del X^{\circ})$ with respect to $I$. 
Then for each $i$, there is a surface $S_i$ in $X^{\circ}$ such that 
$[S_i]$ is an element of the dual basis satisfying $I([S_i], \mu_i)=1$. We may assume that $S_i$ intersects $\del X$ transversely. 
Since $\del M=\phi$, $\del S_i$ is the sum of a longitude of $\del V_{K_i}$ and some meridians. 
Capping off the meridians in $\del S_i$ by the meridian discs, and extending the longitudes of $\del V_{K_i}$ to $K_i$, we obtain a Seifert surface $S'_i$ of $K_i$, and hence $[K_i]=0$ in $H_1(M)$. 
\end{proof}

Following is an important example obtained from a $\Z$-cover. 
\begin{exa}[TLN-cover, \cite{KM2008}]  
Let $L=\cup K_i$ be a link in a 3-manifold $M$, let $\mu_i$ denote the meridian of $K_i$,  
and suppose that $L$ consists of null-homologous components. 
Then a standard $\Z$-cover $\wt{X}\to X=M-L$ is defined by $\tau:\pi_1(X)\to \Z; \forall \mu_i\mapsto 1$ \blue{which induce the zero map on $H_1(M)$ }
and it is called \emph{the TLN-cover over $(M,L)$}. 
We call the inverse system of branched $p$-covers obtained from such a $\Z$-cover \emph{the TLN-$\Zp$-cover} $\wt{M}=\(h_n:M_n\to M\)_n$ over $(M,L)$. 

\blue{We prove in Subsection 4.4 that} 
all the $M_n$ are $\Q$HS$^3$ if and only if 
$M$ is a $\Q$HS$^3$ and the reduced Alexander polynomial \blue{$\Delta_{L,\tau}(t)$} is not divided by any cyclotomic polynomial of $p$-power-th. 
\end{exa}

In order to explain that the notion of a branched $\Zp$-cover gives an essential generalization, we introduce the notion of an isomorphism of them. 
\begin{dfn} 
Branched $\Zp$-covers $\wt{M}=\(h_n:M_n\to M\)_n$ and $\wt{M'}=\{h'_n:M'_n\to M'\}_n$ are \emph{isomorphic} if the following equivalent conditions are satisfied. \\
(1) There is a compatible system $\{f_n:M'_n\congto M_n\}_n$ of 
isomorphisms 
on each layer.\\
(2) There is an isomorphic cover $f_{\red 0}:X'\congto X$ of the exteriors of some links in the exterior, and there is an isomorphism $\iota: \Zp\congto \Zp$ such that $\iota\circ \tau=\tau' \circ f_{\red 0*}$, where $\tau:\pi_1(X)\to \Zp$ and  $\tau':\pi_1(X')\to \Zp$ are the the defining homomorphisms and $f_{\red 0*}:\pi_1(X)\congto \pi_1(X')$ is the induced map.
\end{dfn}
\begin{proof} We prove the equivalence of (1) and (2). 
Suppose an isomorphism $f_{\red 0}:X'\congto X$ is given. 
Let $\tau_n,\tau'_n$ denote the composites of $\tau$, $\tau'$ and the natural surjection $\Zp\surj \Z/p^n\Z$ respectively, 
and consider the following commutative diagram consists of exact rows. 
$$
\xymatrix{
0\ar[r] & \pi_1(X'_n) \ar[r] \ar^{f_{n*}}[d] & \pi_1(X') \ar^{\tau'_n}[r] \ar^{f_{0*}}_{\cong}[d]& \Z/p^n\Z \ar[r] \ar^{\iota_n}[d] &0\\%
0\ar[r]  & \pi_1(X_n)  \ar[r] & \pi_1(X)  \ar^{\tau_n}[r] & \Z/p^n\Z \ar[r] &0}
$$
By diagram chasing, taking $f_n$ is equivalent to taking $\iota_n$, and 
$f_{n*}$ is isomorphic if and only if $\iota_n$ is isomorphic. 
Then the conclusion follows immediately. 
\end{proof}

A branched $\Zp$-cover $\wt{M}$ over $(M,L)$ is isomorphic to one obtained from a $\Z$-cover 
if and only if the defining homomorphism $\tau$ factors some homomorphism $\Z\inj \Zp$. 
If $L$ is a knot and $M$ is a $\Q$HS$^3$, then $\wt{M}$ is always isomorphic to one obtained from a $\Z$-cover. 
However, if $L$ has more than one component, \blue{then} $\wt{M}$ may not: 
\begin{exa}
Let $p\equiv 1$ (mod 4). Then $\sqrt{-1}\in \Zp^*$. 
Let $L=K_1\cup K_2$ be a 2-component link in \blue{$M=S^3$ and put $X=M-L$.} 
Let $\mu_i$ denote the meridians of $K_i$\blue{,} and 
let $\wt{M}$ be a branched $\Zp$-cover over $(M,L)$ defined by 
$\tau: \blue{H_1(X)\to \Zp;}\mu_1\mapsto 1, \mu_2\mapsto \sqrt{-1}$. 
\blue{Since 1 and $\sqrt{-1}$ cannot move into  $\Z$ at the same time by multiplying any unit of $\Zp$, }
$\tau$ does not factor through any homomorphism $\Z\inj \Zp$\blue{. Therefore} 
$\wt{M}$ cannot be obtained from a $\Z$-cover. 

More concretely, let $p=5$. Then $\sqrt{-1}=....1212_{(5)}$ (base 5). 
Let $\wt{M'}$ be another $\Zp$-cover defined by $\tau':\mu_1\mapsto 1, \mu_2\mapsto 12_{(5)}$. Then $h_n$ and $h_n$ are isomorphic for $n\leq 2$. Thus, a $\Zp$-cover can be ``approximate'' by $\Z$-cover as much as we like. 
\end{exa}

\bred{We remark that an analogue of the Hilbert ramification theory for $\Zp$-extensions holds in a natural way. It is an immediate consequence of 
the theory for finite covers given in \cite{Morishita2012} Chapter 5 or \cite{Ueki1}:} 

\bred{ 
Let $\wt{M}=\(h_n:M_n\to M\)_n$ be a branched $\Zp$-cover over $(M,L)$, let $K\subset M$ be a knot in or outside $L$, and let $\wt{K}=\(K_n\)_n$ denote an inverse system of knots over $K$ in $\wt{N}$. 
For each $n$, there is a subgroup of $\Gal(h_n)$ called the inertia group $I_{K_n}$ and the decomposition group $D_{K_n}$ of $K_n$ in $h_n$. 
They satisfy $I_{K_n}<D_{K_n}<\Gal(h_n)$ and control the behavior of $K_n$ as follows: Let $h_n:M_n\to T_n\to Z_n \to M$ denote the corresponding decomposition. Then \violet{(the images of)} $K_n$ is totally branched in $M_n\to T_n$, totally inert in $T_n\to Z_n$, and \red{totally} decomposed in $Z_n\to M$. Moreover, since each subgroup of $\Gal(h_n)\cong \Z/p^n\Z$ is equal to $\Gal(h_{m,n})$ for some $m\leq n$, $T_n=M_{n_1}$ and $Z_n=M_{n_2}$ for some $n_2\leq n_1\leq n$.}  

\bred{\violet{Note that $\(I_{K_n}\)_n$ and $\(D_{K_n}\)_n$ form surjective systems.} We define \emph{the inertia group} $I_\wt{K}$ and \emph{the decomposition group} $D_\wt{K}$ of $\wt{K}$ as their inverse limits. \violet{Since $\varprojlim \Gal(h_n)\cong \Zp$,} they are open subgroups with $I_{\wt{K}}<D_{\wt{K}}<\varprojlim \Gal(h_n)$, and they control the behavior of $\wt{K}$ in $\wt{M}$ in a similar way to the case of finite covers. 
\red{Since each open subgroup $G'$ of $\varprojlim \Gal(h_n)\cong \Zp$ satisfies $\varprojlim \Gal(h_n)/G'\cong \Gal(h_m)$  for some $m\in \N$}, we have the following.} 

\begin{prop}
\bred{
Let $\wt{M}=\(h_n:M_n\to M\)_n$ be a branched $\Zp$-cover over $(M,L)$, let $K\subset M$ be a knot in or outside $L$, and let $\wt{K}=\(K_n\)_n$ denote a surjecitve system of knots over $K$ in $\wt{N}$. 
Then $\wt{K}$ satisfies one of the following: (i) infinitely branched, finitely inert, and finitely branched, (ii) unbranched, infinitely inert, and finitely decomposed, (iii) unbranched, \violet{non-inert}, and totally decomposed.} 
\end{prop} 

\violet{We say a branched $\Zp$-cover $\wt{M}=\(h_n:M_n\to M\)_n$ over $(M,L)$ is \emph{properly branched} if it satisfies the following equivalent conditions: (1) Some $h_n$ is properly branched over some non-empty link, (2) $\wt{M}$ is infinitely branched over some non-empty link, (3) $\tau(\mu_i)\neq 0$ for some meridian $\mu_i$ of $L$. }  

\subsection{The Iwasawa type formula and Sakuma's exact sequence}
In this subsection, we prove the Iwasawa type formula for a branched $\Zp$-cover of $\Q$HS$^3$. Moreover, we state a $p$-adic variant of Sakuma's exact sequence. Then we deduce an alternative \blue{direct} proof of the Iwasawa type formula, and also an assertion on a direct limit module.  

\blue{For any manifold $M$, we denote $H_1(M)_{[p]}:= H_1(M,\Zp)=H_1(M)\otimes \Zp$. If $M$ is a $\Q$HS$^3$, then $H_1(M)_{[p]}$ is identified with the $p$-torsion subgroup of $H_1(M)$.} 
A generalization of the analogue of Iwasawa's class number formula (\cite{HMM2006}, \cite{KM2008}) describes the 
behavior of $p$-torsions in a branched $\Zp$-cover of $\Q$HS$^3$:  
\begin{thm}[the Iwasawa type formula] \label{ITF}
Let $\wt{M}=\(h_n:M_n\to M\)_n$ be a branched $\Zp$-cover of $\Q$HS$^3$ over $(M,L)$. 
Then for the $p$-torsion subgroups $H_1(M_n)_{[p]}$ of \blue{the} 1st homology groups, 
there are some $\lambda, \mu \in \N, \nu\in \Z$ and $n_0 \in \N$ such that for $n>n_0$,  
$$\#H_1(M_n)_{[p]} = p^{\lambda n+\mu p^n+\nu}.$$ 
\end{thm}
\begin{dfn}
These $\lambda,\mu,$ and $\nu$ are called \emph{the Iwasawa invariants} of $\wt{M}$. 
\magenta{They are sometimes denoted as}  
\violet{$\lambda_{\wt{M}},\mu_{\wt{M}}$, and $\nu_{\wt{M}}$.}  
\end{dfn}

Since a branched $\Zp$-cover can be approximated by $\Z$-covers as much as we like, the proof of this formula comes down to the case obtained from a $\Z$-cover (\cite{KM2008}). 

\begin{proof} 
For each $n_1\in \N$, let $\tau_{n_1}:H_1(X)\surj \Z/p^{n_1}\Z$ be the defining homomorphism of $M_{n_1}\to M$. 
Since $\tau_{n_1}$ lifts to the homomorphism $\tau$ to $\Zp$, 
$H_1(X)_{\rm tor}<\Ker \tau_{n_1}$, and 
$\tau_{n_1}$ sends a torsion-free element $b\in H_1(X)$ to a unit of $\Zp$. 
By composing an automorphism of $\Z/p^{n_1}\Z$, we may assume that $\tau_{n_1}(b)=1$. 
Then $\tau_{n_1}$ lifts to a surjective homomorphism $\wt{\tau_{n_1}}:H_1(X_L)\surj \Z$ obtained as follows: 
take a $\Z$-basis of $H_1(X)_{\rm free}:=H_1(X)/H_1(X)_{\rm tor}$, and define $H_1(X)_{\rm free}\surj \Z$ by sending basis to the integers with the same presentation of the images by $\tau_{n_1}$, and let $\wt{\tau_{n_1}}$ be the composite with $H_1(X)\surj H_1(X)_{\rm free}$. 

\blue{Note that} there is  $n_0$ independent of $n_1$ and 
the Iwasawa type formula holds for $n_0<n<n_1$.  
Indeed, $n_0$ is determined by $H_1(M)$ and the $p$-adic valuations of the images of $\tau$ (\cite{KM2008}). 
Therefore, for sufficiently large $n_1$, the Iwasawa invariants of branched covers $\Zp$ defined by $\wt{\tau_{n_1}}$ and $\tau$ coincide. 
\end{proof} 
In the proof of the \blue{Iwasawa type} formula by Kadokami--Mizusawa in \cite{KM2008}, 
Sakuma's exact sequence (\cite{Sakuma1981} Section 4, \cite{KM2008} Lemma 3.4) played a key role. 
A $p$-adic variant of this sequence is stated in the following.  

\blue{Let $L=\sqcup K_i$ be a link in a 3-manifold $M$, put $X=M-L$, and let $\mu_i \in H_1(X)$ denote the meridian of $K_i$. Let $\wt{M}$ be a branched $\Zp$-cover over $(M,L)$ defined by $\tau:H_1(X)\to \Zp$ \green{and suppose that it is properly branched}. 
Then $h_{n,n+1 *}: H_1(M_{n+1})\to H_1(M_n)$ is surjective for any $n\gg 0$ (e.g., \cite{Ueki1} Theorem 6). 
We define the Iwasawa module of $\wt{M}$ by $\mca{H}:=\varprojlim H_1(M_n)_{[p]}$. 
We fix an identification $\Lambda=\Zp[[T]]\cong \Zp[[\wh{\langle t\rangle}]] = \varprojlim \Zp[t]/(t^{p^n}-1); 1+T\LR t$. Then $\mca{H}$ is a $\Lambda$-module. 
For each $n$, $H_1(M_n)_{[p]}$ is a $\Zp[t]/(t^{p^n}-1)$-module, and hence also is a $\Lambda$-module. Put $\nu_{p^n}=1+t+...+t^{p^n-1}$. 
Now we have the following:} 

%
\begin{prop}[Sakuma's exact sequence]\label{Sakuma} Let $\wt{M}$ be a branched $\Zp$-cover 
and let the notation be as above. \green{Suppose that $\wt{M}$ is properly branched.}\\ 
(1) If $\tau(\mu_i)\neq 0 \mod p$ for every $i$, \blue{equivalent to say, if $\wt{M}$ is totally branched over $L$, }then for each $n$ we have  an exact sequence 
$$H_1(M)_{[p]}\to H_1(M_n)_{[p]}\to \mca{H}/\nu_{p^n}\mca{H}  \to 0.$$
\noindent 
(2) In any case, there is some $n_0$ such that for any $n>n_0$ we have an exact sequence 
$$H_1(M_{n_0})_{[p]}\to H_1(M_n)_{[p]}\to \mca{H}/\nu_{p^{n-n_0}}\mca{H} \to 0.$$
\end{prop} 
The proof will be given in the subsequent subsection. 
Note that \blue{if $H_1(M)$ or $H_1(M_{n_0})$ is finite, then its image} 
by the transfer maps in $H_1(M_n)$'s \blue{is} 
constant for $n\gg 0$. 
Combining Proposition \ref{Sakuma} and the next lemma, the Iwasawa type formula (Theorem \ref{ITF}) follows immediately from the structure theorem of compact $\Lambda$-modules (Lemma \ref{Lambda} (3)), 
similarly to the case of number theory in Section 3. 
\begin{lem} \blue{Suppose that $\wt{M}$ consists of $\Q$HS$^3$'s. Then the} 
Iwasawa module 
$\mca{H}=\varprojlim H_1(M_n)_{[p]}$ of $\wt{M}$ is a finitely generated torsion $\Lambda$-module. 
\end{lem}
\begin{proof} 
Since $\mca{H}$ is an inverse limit of $\Lambda$-modules with finite orders, it is a compact $\Lambda$-module. 
In the composite $H_{1}(M_n)_{[p]}/h_n^!(H_{1}(M_{1})_{[p]})$ $\cong$ $\mca{H}/\nu_{p^{n}}\mca{H}$ $\surj$ $\mca{H}/(p,T)$, 
the right hand term is finite. 
Hence the module $\mca{H}$ is finitely generated by Nakayama's lemma (Lemma \ref{Lambda} (1)). 
Moreover, since $\mca{H}/\nu_{p^{n}}\mca{H}$ is a finite group, $\mca{H}$ is a torsion $\Lambda$-module by Lemma \ref{Lambda} (3). 
\end{proof} 
\blue{We note that $\mca{H}$ 
can be a finitely generated torsion $\Lambda$-module even if $\wt{M}$ does not consist of $\Q$HS$^3$'s (See Subsection 4.4).}

Finally, we prove an assertion on a direct limit module by the transfer maps. 
\blue{It plays an important role in the proof of Kida's formula in Section 7.}  
\begin{prop}\label{injlim} Let $\wt{M}$ be a branched $\Zp$-cover of $\Q$HS$^3$ and let $H_1(\wt{M})_{[p]}:=\varinjlim H_1(M_n)_{[p]}$ denote the direct limit by the transfer maps. Then, there is an isomorphism of discrete $\Lambda$-modules $H_1(\wt{M})_{[p]}\cong (\Qp/\Zp)^{\lambda_{\wt{M}}}\oplus A'$,
where $A'$ is a bounded $\Lambda$-module, namely, there is some $a\in \N$ such that $p^aA'=$0, and $A'=0$ holds if and only if $\mu_{\wt{M}}=0$. 
\end{prop} 

\begin{proof} 
\blue{In general, let $E$ be a $\Lambda$-module. If $E\sim \Lambda/(p^m)$, then $\varinjlim E/ E/\nu_{p^n} E$ by $\nu_{p^{n+1}}/\nu_{p^n}$ is a bounded infinite group. 
If $E\sim \Lambda/(p^m)$, then $\varinjlim E/ E/\nu_{p^n} E\cong (\Qp/\Zp)^\lambda$. 
Let $E\sim E'$ be a pseudo isomorphism of $\Lambda$-modules. Then their direct limits by $\nu_{p^{n+1}}/\nu_{p^n}$ are isomorphic to each other.}


\blue{Now we} consider the following direct system for the exact sequences of Proposition \ref{Sakuma}. 
$$\xymatrix{
H_1(M)_{[p]} \ar^{h_{n+1}^!}[r] &H_1(M_{n+1})_{[p]}\ar[r] 
&\mca{H}/\nu_{p^{n+1}}\mca{H} \ar[r] &0\\ 
H_1(M)_{[p]} \ar^{h_n^!}[r] \ar@{=}[u] &H_1(M_n)_{[p]} \ar[r] \ar^{h_{n+1,n}^!}[u] &\mca{H}/\nu_{p^n}\mca{H} \ar[r] \ar^{\nu_{p^{n+1}}/\nu_{p^n}}[u] &0
}$$ 
\blue{Since} the direct limit functor is exact, there is an isomorphism 
$H_1(\wt{M})_{[p]}/h_\infty^!(H_1(M)_{[p]})$ $\cong (\Qp/\Zp)^{\lambda_{\wt{M}}}\oplus A''$,
\blue{where} 
$A''$ is a bounded $\Lambda$-module, and $A''=0$ if and only if $\mu_{\wt{M}}=0$. 
\blue{Since} $h_\infty^!(H_1(M)_{[p]})$ is a finite group, 
there is an isomorphism $H_1(\wt{M})_{[p]}\cong (\Qp/\Zp)^{\lambda_{\wt{M}}}\oplus A'$ with a bounded module $A'$. 
If $\mu_{\wt{M}}=0$, then by Lemma \ref{Lambda} (4), the $p$-ranks of $H_1(M_n)_{[p]}$ are bounded, and hence the $p$-ranks of $h_n^!(H_1(M)_{[p]})$ are also bounded. 
Moreover, since $\tau:H_1(X)\to \Zp$ sends torsions to zero, 
$h_{n,n+1*}:H_1(M_{n+1})\to H_1(M_n)$ is surjective on the torsion subgroup. 
Thus, for sufficiently large $n$, a map $h_{n,n+1}^!\circ h_{n,n+1*}$ on $H_1(M_{n+1})$ is multiplication by $p$, 
and $h_{n}^!(H_1(M)_{[p]})$ is contained in $pH_1(M_{n+1})_{[p]}$.  Therefore $h_{\infty}^!(H_1(M)_{[p]})$ has to be a subgroup of a divisible group, and there is an isomorphism 
$H_1(\wt{M})_{[p]}\cong (\Qp/\Zp)^{\lambda_{\wt{M}}}$. 
Conversely, if this isomorphism holds, then $\mu_{\wt{M}}=0$ clearly holds. \end{proof}

\subsection{Proof of Sakuma's exact sequence (Proposition \ref{Sakuma})}
We prove Proposition \ref{Sakuma} by modifying the argument in \cite{KM2008}, Chapter 3. 

The assertion (2) can be reduced to (1). Indeed, 
\blue{consider the branched cover $h_{n,n+1}:M_{n+1}\to M_n$ of degree $p$, 
let $K$ be a component of $h_n^{-1}(L)$ in $M_n$, let $\mu$ denote its meridian, 
and $\wt{\mu}$ \bred{a meridian over $\mu$} in $M_{n+1}$. 
If $h_{n,n+1}$ is branched along $K$, then $h_{n,n+1*}(\wt{\mu})=p\mu$. If otherwise, then $h_{n,n+1*}(\wt{\mu})=\mu$. 
Therefore, if we put $n_0=\max \(v_p(\tau(\mu_i))\)_i$ 
\bred{where $v_p(x)$ denote the $p$-adic valuation of a number $x$}, 
then $\tau_{n_0}:H_1(X_{n_0})\to p^{n_0}\Zp\congto \Zp$ sends 
every meridian $\mu_i$ of $h_{n_0}^{-1}(L)$ in $H_1(X_{n_0})$ 
to an element whose image by mod $p$ is non-trivial. 
Thus, by replacing the base space by $X_{n_0}$, we can reduce (2) to (1).} 


\blue{In order to prove the assertion (1), we prepare several lemmas.}  
We put $\Lambda_0=\Z[\langle t\rangle]$. 
Then every $H_1(M_n)$ is a $\Lambda_0=\Z[\langle t\rangle]$-module. 

\begin{lem} \label{lem-mn} 
\blue{Suppose $n<m$. If $\wt{M}$ is totally branched over $L$, then there is 
\bred{a natural} exact sequence of $\Lambda_0$-modules:} 
$$p^{m-n}\Z/p^m \Z \to H_1(X_m)/\nu_{p^n}H_1(X_m) \to H_1(M_n)/h_n^!(H_1(M))\to 0.$$
\blue{In addition, for the meridian module $\langle\wt{\mu_i}\rangle :=\Ker(H_1(X_m)\to H_1(M_m))$, 
there is 
\bred{a natural} exact sequence:} 
$$0\to \langle\wt{\mu_i}\rangle/\nu_{p^n} \to H_1(X_m)/\nu_{p^n}H_1(X_m)
\to H_1(M_m)/\nu_{p^n} H_1(M_m) \to 0.$$ 
\blue{By taking $\otimes \Zp$, similar exact sequences of $\Lambda$-modules are obtained.} 
\end{lem} 

\begin{proof} 
Put $G:=\Gal(h_m)$ for the subcover $h_m:X_m\to X$ of degree $p^m$. Then there is an exact sequence 
$1\to \pi_1(X_m)\to \pi_1(X)\to G \to 1$, and 
the Hochschild--Serre spectral sequence (\cite{Brown}) yields 
an exact sequence  
$H_2(G)\to H_1(\pi_1(X_m))_{G} \to H_1(\pi_1(X)) \to H_1(G)\to 0$. 
Since $G=\blue{\langle t \mid t^{p^m}\rangle}\cong \Z/p^{m}\Z$ is a finite cyclic group, 
by applying the Hurewicz isomorphism, we have 
$H_2(G)=0$, 
$H_1(\pi_1(X_m))_G=(\pi_1(X_m)^{\rm ab})_G\cong H_1(X_m)_G=H_1(X_m)/(t-1)H_1(X_m)$, 
$H_1(\pi_1(X))=\pi_1(X)^{\rm ab}\cong H_1(X)$, and $H_1(G)=G\cong \Z/p^m\Z$.  
Therefore we obtain an exact sequence $0\to (t-1)H_1(X_m)\to H_1(X_m) \to H_1(X) \to \Z/p^m\Z \to 0$. 
In a similar way, for a subcover $h_{n,m}:X_m\to X_n$, 
an exact sequence 
$1\to \pi_1(X_m)\to \pi_1(X_n)\to \Gal(h_{n,m})\to 1$ 
yields an exact sequence 
$0\to (t^{p^n}-1)H_1(X_m)\to H_1(X_m) \to H_1(X_n) \to p^n(\Z/p^m\Z) \to 0$. 
Therefore, a commutative diagram with exact rows 
$$\xymatrix{
C_*(X_m) \ar^{t^{p^n} -1}[r] &C_*(X_m) \ar^{h_{n,m*}}[r] &C_*(X_n)\ar[r] & 0\\
C_*(X_m) \ar^{t-1}[r] \ar@{=}[u] &C_*(X_m) \ar^{h_{m*}}[r] \ar^{\nu_{p^n}}[u] &C_*(X)\ar[r] \ar^{h_n^!}[u] &0
}$$
yields the following commutative diagram with exact rows 
$$\xymatrix{
H_1(X_m) \ar^{t^{p^n}-1}[r] &H_1(X_m)\ar[r] &H_1(X_n)\ar^{\wt{\tau}_n}[r]  &p^n(\Z/p^m\Z) \ar[r] &0 \\
H_1(X_m) \ar@{=}[u] \ar^{t-1}[r] &H_1(X_m) \ar[r] \ar^{\nu_{p^n}}[u] &H_1(X)\ar^{\wt{\tau}}[r] \ar^{h_n^!}[u] & \Z/p^m\Z \ar[r] \ar^{\times p^n}[u]&0. 
}$$
Since $\nu_{p^n}:(t-1)H_1(X_m)\to (t^{p^n}-1)H_1(X_m)$ is surjective, 
there is an isomorphism 
$\Coker(\nu_{p^n}:H_1(X_m)\to H_1(X_m))\cong \Coker(\nu_{p^n}: H_1(X_m)/(t-1)H_1(X_m)\to$ $H_1(X_m)/(t^{p^n}-1)H_1(X_m)$. 
\blue{By} the Snake lemma, \blue{we obtain} 
 an exact sequence $p^{m-n}\Z/p^m \Z \to H_1(X_m)/\nu_{p^n}H_1(X_m) \to H_1(X_n)/h_n^!(H_1(X))\to 0$. 
 
\blue{Now we consider the following commutative diagram with exact rows. 
$$\xymatrix{
H_2(M_n, X_n) \ar^{\del_n}[r] &H_1(X_n) \ar[r] &H_1(M_n) \ar[r] &0\\
H_2(M, X) \ar[u] \ar^{\del_1}[r] &H_1(X) \ar^{h_n^!}[u] \ar[r] &H_1(M) \ar^{h_n^!}[u] \ar[r] &0 
}$$
By the assumption of (1) that $\wt{M}$ is totally branched over $L$, the transfer map $h_n^!$ is surjective on the meridians. Hence we have an isomorphism $h_n^!: \Im \del_1\to \Im \del_n$. 
By the snake lemma, we obtain an isomorphism $H_1(X_n)/h_n^!(H_1(X))\congto H_1(M_n)/h_n^!(H_1(M))$ on the $\Coker$ of two vertical morphisms on the right side.}  

Thus we have obtained an exact sequence of $\Lambda_0$-modules. \end{proof}
%

\blue{Next, we} 
recall some facts on inverse limits. 
We say \violet{that} an inverse system $\(A_n\)_n$ \blue{of abelian groups} satisfies \emph{the Mittag-Leffler condition} (ML-condition) if 
``for any $n$ there exists some $n'$ such that for any $n''>n'$, $\Im(A_{n''}\to A_n)$ is constant''. 
If $\(A_n\)_n$ consists of finite \violet{groups}, or if its morphisms are all surjective, 
then it satisfies ML-condition. 
An inverse system $\(A_n\)_n$ is said to be \emph{ML-zero} if it satisfies a modified ML-condition in which ``constant'' is replaced by ``zero''. 
\violet{Recall $\Lambda=\Zp[[T]]$. A \emph{profinite $\Lambda$-module} is one obtained as the inverse limit of discrete $\Lambda$-modules. A \emph{homomorphism of profinite $\Lambda$-modules} is a continuous $\Lambda$-module homomorphism. It is always a closed map, because it is a continuous map from a compact space to a Hausdorff space. 
Note that profinite $\Lambda$-modules form an abelian category. By \cite{Jannsen1988} \S 1, we have the following lemma.}  


\begin{lem}\label{ML} Let $\(A_n\)_n, \(B_n\)_n, \(C_n\)_n$ be inverse systems of \violet{profinite $\Lambda$-modules.}\\ 
(1) Then, an exact sequence  $0\to A_n \to B_n \to C_n \to 0$ of inverse systems yields an exact sequence 
$0\to \varprojlim A_n \to \varprojlim B_n \to \varprojlim C_n \to \varprojlim^1 A_n$. 
If $\(A_n\)_n$ satisfies \blue{the} ML-condition, then $\varprojlim^1 A_n=0$ holds.\\ 
(2) If $\(A_n\)_n$ is ML-zero, then an exact sequence $A_n \to B_n \to C_n \to 0$ yields 
an isomorphism $\varprojlim B_n \congto \varprojlim C_n$.\\ 
(3) \violet{Suppose that $\(A_n\)_n$ is a surjective system and let $\(f_n:A_n\to A_n\)_n$ be a family of endomorphisms commutative with the system.}  
Put $\mca{A}=\varprojlim A_n$ and let $f:\mca{A}\to \mca{A}$ denote the induced map. Then $\mca{A}/f(\mca{A})\congto \varprojlim A_n/f_n(A_n)$.
\end{lem}

\begin{proof} 
\violet{Here we prove (3). Since each $f(\mca{A})\to f_n(A_n)$ is a surjection, 
the inclusion map $\iota:f(\mca{A})\inj \varprojlim_n f_n(A_n)$ is a continuous homomorphism with dense image. Since $\iota$ is a closed map, $\iota$ is an isomorphism and we have $f(\mca{A})= \varprojlim_n f_n(A_n)$.} 
Now consider \violet{exact sequences} $0\to f_n(A_n) \to A_n \to A_n/f_n(A_n)\to 0$ compatible with the inverse system. Since $\(f_n(A_n)\)_n$ is a surjective system, (1) yields an exact sequence  $0\to \varprojlim_n f_n(A_n) \to \mca{A} \to \varprojlim_n A_n/f_n(A_n)\to 0$. Thus we obtain $\mca{A}/f(\mca{A})\congto \varprojlim A_n/f_n(A_n)$. 
\end{proof}

\blue{Now we give a proof of Proposition 4.11 (1):}  

\begin{proof}[Proposition \ref{Sakuma} \blue{(1)}, Sakuma's exact sequence]　
We consider the inverse systems with respect to $m$ in Lemma \ref{lem-mn}. 
In the first exact sequence, the \red{transition morphisms on} the first terms $L_m:=p^{m-n}\Z/p^m\Z \red{\cong \langle t^{p^{m-n}} | t^{p^m}\rangle}$ are 
\red{ $\{ \mod p^m: L_{m+1}\to L_m \}_m$. 
Hence} 
if $m'>\red{m+n},$ then $L_{m'}\to L_m$ is the zero map. 
Thus the first term is ML-zero, 
\red{and Lemma \ref{ML} (2) yields} 
an isomorphism $\varprojlim_m (H_1(X_m)_{[p]}/\blue{\nu_{p^n}} H_1(X_m)_{[p]})$ $\red{\congto}  H_1(\blue{M_n})_{[p]}/h_n^!(H_1(M)_{[p]})$. 

In the second exact sequence in Lemma \ref{lem-mn},  
since the 
\red{transition maps on} the finite $p$-torsion \red{abelian} groups $\langle\mu_i\rangle/\nu_{p^n}$ are the multiplication by $p$, 
\red{these terms} also satisfy ML-zero\blue{. Therefore by} 
Lemma \ref{ML} (2), there is an isomorphism 
$\varprojlim_m(H_1(X_m)_{[p]}/\nu_{p^n} H_1(X_m)_{[p]})\cong \varprojlim_m(H_1(M_m)_{[p]}/\nu_{p^n} H_1(M_m)_{[p]})$. 
\blue{Since $\(H_1(M_m)_{[p]}\)_m$ is a surjective system of pro-$p$ \red{abelian} groups, 
Lemma \ref{ML} (3) yields an isomorphism 
$\varprojlim_m(H_1(M_m)_{[p]}/\nu_{p^n} H_1(M_m)_{[p]})\cong \mca{H}/\nu_{p^n}\mca{H}$.   
}
\end{proof} 



\subsection{Remarks on non-$\Q$HS$^3$ cases}

\blue{In this subsection, we briefly discuss further generalization of Iwasawa type formula for non-$\Q$HS$^3$ cases.} %



\begin{lem} (1) Let $E\sim \oplus_i \Lambda/(p_i^{m_i}) \oplus \oplus_j \Lambda/(f_j^{e_j})$ be a pseudo isomorphism from a finitely generated torsion $\Lambda$-module to the normal form. Then $E/\nu_{p^n}E$ is a infinite group if and only if $f_i$ are not $p^{n'}$-th cyclotomic polynomial for any $\red{1\leq} n'\leq n$.\\
(2) Let $f$ be \red{a $p$-power-th cyclotomic polynomial in $\Z[t]$}, which is an irreducible Weierstrass polynomial in $\Zp[[T]]$, and put $E=\Lambda/(f^e)$. Then the following are equivalent: 
that $E/\nu_{p^n}E$ is constant for $n\gg 0$, that $(f^e, \nu_{p^n})$ is constant for  $n\gg 0$, and that $e=1$. \end{lem}

\begin{proof} Note that $\nu_{p^n}=\nu_{p^n}(t)=(t^{p^n}-1)/(t-1)=\nu_{p^n}(1+T)$ in $\Z[t]\subset \Zp[[T]]; t\mapsto 1+T$ is the product of all the $p^{n'}$-th cyclotomic polynomials for $1\leq n'\leq n$. 

(1) A module of the form $(\Lambda/(p^{m_i}))/\nu_{p^n}(\Lambda/(p^{m_i}))$ is finite in any case. 
By a general fact that monic polynomials $g,h\in \Z[t]$ with positive degrees satisfy $\#\Z[t]/(g,h)<\infty$ if and only if they have no common divisor, 
a module of the form $(\Lambda/(f_j^{e^j}))/\nu_{p^n}(\Lambda/(f_j^{e^j}))=\Lambda/(f_j^{e_j}, \nu_{p^n})$ is finite if and only if the monic polynomials $f_j^{e_j}$ and $\nu_{p^n}$ have no common divisor, that is, $f_j$ is not the $p^{n'}$-th cyclotomic polynomial for $1\leq n'\leq n$. 

(2) Let $f$ be the $p^n$-th cyclotomic polynomial in $\Z[t]\subset \Zp[[T]]$. 
Then $(f,\nu_{p^{n'}})=(f)$ holds for any $n'\geq n$. 
We suppose both $e>1$ and that $E/\nu_{p^{n'}}E$ is constant for $n'\gg 0$, and lead a contradiction. 
By these hypothesises, we have $(f^e, \nu_{p^n})=(f)(f^{e-1}, \nu_{p^n}/f)$, 
and there is some $n'\geq n$ with $(f^{e-1}, \nu_{p^{n'}}/f)=(f^{e-1}, \nu_{p^{n'+1}}/f)$. 
By taking $\mod f$, we have 
$(\nu_{p^{n'}}/f) = (\nu_{p^{n'+1}}/f)$ in $\Lambda/(f)$.  
Let $\alpha=\alpha_1,\ldots,\alpha_{\blue{\varphi(p^n)}}$ denote the roots of $f$ in an algebraic closure $\ol{\Z}_p$ of $\Zp$, where \blue{$\varphi(n)$} is the Euler function. 
Then we have a standard injective homomorphism $\Lambda/(f)\inj \ol{\Z}_p^{\blue{\varphi(p^n)}}$, 
and $((\nu_{p^{n'}}/f)(\alpha))= ((\nu_{p^{n'+1}}/f)(\alpha))$ in $\ol{\Z}_p$.  
However, the $p^{n'+1}$-th cyclotomic polynomial $q_{p^{n'+1}}$ satisfies $\nu_{p^{n'+1}}/f=q_{p^{n'+1}}\nu_{p^{n'}}/f$ and $(q_{p^{n'+1}}(\alpha))=(p^{n'+1-n})\neq (1)$. 
Thus we have $((\nu_{p^{n'}}/f)(\alpha))\neq ((\nu_{p^{n'+1}}/f)(\alpha))$, and hence contradiction. 
\end{proof}

\begin{thm} 
Let $\wt{M}$ be a branched $\Zp$-cover of 3-manifolds, 
\blue{and suppose that $\mca{H}=\varprojlim H_1(M_n)_{[p]}$ is a finitely generated torsion $\Lambda$-module.} Then\\ 
\noindent 
(1) The $n$-th cover $M_n$ is $\Q$HS$^3$ if and only if $M$ is a $\Q$HS$^3$ and the \blue{characteristic} polynomial $char_{\mca{H}}$ of $\mca{H}$ is not divided by $p^{n'}$-th cyclotomic polynomial for $n'\leq n$. \\
(2) If every $p$-power-th cyclotomic polynomial contained in  $char_{\mca{H}}$ is of exponent 1, then the $p$-torsion subgroups  $H_1(M_n)_{[p]}$ satisfy the Iwasawa type formula. \end{thm}

\begin{proof} 
Let $\mca{H}\sim \oplus_i \Lambda/(p_i^{m_i}) \oplus \oplus_j \Lambda/(f_j^{e_j})$ be a pseudo isomorphism to a normal form, 
and put $E_j:=\Lambda/(f_j^{e_j})$ for each $j$. 
Then $E_j/\nu_{p^n}E_j$ contains a free $\Zp$-module if and only if 
$f_j$ is the $p^{n'}$-th cyclotomic polynomial for some $n'\leq n$. 
If $E_j/\nu_{p^n}E_j$ is constant for $n\gg 0$ for every $j$ such that $f_j$ is a $p$-power-th cyclotomic polynomial, then by considering the direct sum of factors which do not correspond to such $j$'s, we have the same augment as before. Therefore the previous lemma yields this theorem.
\end{proof}
\begin{exa} 
\blue{Let $M$ be a $\Q$HS$^3$, let $L=\sqcup K_i$ be a link with null homologous $d$-components in $M$, put $X=M-L$, and let $\mu_i\in H_1(X)$ denote the meridian of $K_i$. 
Let $v_1,...,v_d$ be units of $\Zp$, let $\tau: H_1(X)\to \Zp$ be a homomorphism defined by $\mu_i \mapsto v_i$, and let $\wt{M}$ be the branched $\Zp$-cover defined by $\tau$. 
Then $\wt{M}$ is totally branched over $L$. 
Let $\Delta_{L,\tau}(t)=\Delta_{L}(t^{v_1},\ldots,t^{v_d})$ denote the $p$-adic reduced Alexander polynomial in $\Zp[[\wh{\langle t\rangle}]]$. 
Since we can approximate $\wt{M}$ by $\Z$-covers as much as we can, 
from a well-known fact for $\Z$-covers and \cite{KM2008} Theorem 3.3 (a variant of the Mayberry--Murasugi formula \cite{MM1982} or Porti's result \cite{Porti2004}), 
we can easily deduce isomorphisms $\mca{H}\congto \Lambda/(\Delta_{L,\tau}(t))$ and  $H_1(M_n)/h_n^{!}(H_1(M))_{[p]}\congto \mca{H}/\nu_{p^n}\mca{H}$, and 
$||H_1(M_n)/h_n^{!}(H_1(M))||_p=\prod_{\zeta^n=1} |\Delta_{L,\tau}(\zeta))|_p$. 
(For a group $G$, we put $|G|=\#G$ or zero 
according as $G$ is finite or infinite.  For a number $x$, $|x|_p$ denotes the $p$-adic norm.)}



Typical examples with cyclotomic Alexander polynomial is torus knots.  
\blue{Only for here, the subscription $n$ means the covering degree over $S^3$. 
If $L=K$ is the trefoil in $M=S^3$,}  
then $\Delta_K(t)=t^2-t+1$ is the 6-th cyclotomic polynomial, and 
$H_1(M_6)=\Z^2$, $H_1(M_3)=(\Z/2\Z)^2$, $H_1(M_2)=\Z/3\Z$ are known. 
For 3-fold cover $M_6\to M_2$ and double-cover $M_6\to M_3$, 
the maps on $H_1$ are not surjections on 3-torsions and 2-torsions respectively. However, for $n>0$, $M_{2\cdot 3^{n+1}}\to M_{2\cdot 3^n}$ and $M_{3\cdot 2^{n+1}}\to M_{3\cdot 2^n}$ induce surjections on 3-torsions and 2-torsions respectively, and satisfy the Iwasawa type formula with trivial invariants. If we take the connected sum with the trefoil and the figure eight knot, then we obtain a non-trivial example. \end{exa}

\begin{rem}  
\blue{For a branched $\Z$-cover over a link $L$ in $S^3$, we have the balance formula among 
\emph{the $p$-adic Mahler measure} of the Alexander polynomial, the 
the Iwasawa $\mu$-invariant, and the $p$-adic entropy (\cite{Ueki4}).} 

\blue{We expect further study (i) with use of higher Alexander polynomial (\cite{SW2002M}, \cite{Le2014}), (ii) for graph-branched cases (\cite{Porti2004}), (iii) about the asymptotic formula related to hyperbolic volumes (\cite{BV2013}, \cite{Le2014}).} 

\end{rem}

\section{
\bred{Morphisms} of branched $\Zp$-covers}

In this section, we introduce an analogue object of an extension of $\Zp$-fields, 
and discuss a condition for $\mu=0$. 
We also define \emph{chains} of branched $\Zp$-covers by the direct limits with respect to the transfer maps, and calculate Tate cohomologies for 
\bred{an equivariant Galois morphism} of branched $\Zp$-covers. 

\subsection{Extensions of $\Zp$-fields}
Let $k_\infty/k$ be a $\Zp$-extension over a finite number field 
and let $F_\infty/k_\infty$ be a $p$-extension. 
Then there is  a $p$-extension $F/k$ such that $F_\infty/F$ is a $\Zp$-extension and $F_n=Fk_n$ for each $n$. 
In such a situation, behaviors of the Iwasawa invariants are studied (\cite{Iwasawa1973}, \cite{Iwasawa1981}). 

Recall that for a $\Zp$-extension $k_\infty/k$, the value of $\lambda$ and the property whether $\mu=0$ or not are independent of the choice of base field $k$, while the values of $\mu$ and $\nu$ depend. 
Let $k_1$ the 1st middle field of $k_\infty/k$. Then a $\Zp$-extension $k_\infty/k_1$ has $\mu$ and $\nu$ different from those of $k_\infty/k$. 
In this case, if we put $F_n:=k_{n+1}$, then we have $F_n=Fk_n$. 
However, we do not assume such a case as an essential extension of $\Zp$-fields. 

\subsection{
\bred{Morphisms of branched $\Zp$-covers}} 
We define an analogue notion of an extension of $\Zp$-fields \blue{as follows:}  

\begin{dfn}[branched cover of branched $\Zp$-covers] 
\blue{Let $\wt{M}$ and $\wt{N\,}\!$ be branched $\Zp$-covers over $(M,L)$ and $(N,L')$ respectively. Then \bred{a \emph{morphism} $f:\wt{N\,}\!\to \wt{M}$ of branched $\Zp$-covers} is a compatible system of branched covers on each layer, that is, 
a family $\( f_n:N_n \to M_n\)_n$ of branched covers commutative with the \red{transition maps.} 
If 
every $f_n$ is Galois, 
\red{then} we say this 
\bred{morphism} is \emph{Galois}. We put $G_n=\Gal(f_n)$. If all the induced maps $G_n\to G_0$ are isomorphisms, then we call this 
\bred{morphism \emph{equivariant Galois}, or} a \emph{branched Galois cover of branched $\Zp$-covers}, and we write $\Gal(f)=G_0$.}
\end{dfn}


\blue{
Let $f:\wt{M}\to \wt{N\,}\!$ be a 
\bred{morphism of branched $\Zp$-covers with notation} as above. 
We can easily check the following facts by diagram chasing, using Proposition 4.8:  
If $\wt{M}$ and $\wt{N\,}\!$ are properly branched over $L$ and $L'$, then $L'=f_0^{-1}(L)$. 
Let $\ol{S}_n\subset M_n$ denote the branch link of $f_n$. 
If $f$ is \bred{equivariant Galois}, 
then we have $\ol{S}_n\subset h_n^{-1}(\ol{S}_0)$.}  
\violet{If in addition if $\ol{S}_0$ is unbranched in $\wt{M}$, then $\ol{S}_n=h_n^{-1}(\ol{S}_0)$ and $L\cap \ol{S}_0=\phi$.}  
\blue{
We denote the restrictions to the exteriors of the preimages of \violet{$L\cup \ol{S}_0$} by the same letters 
$h_n:X_n\to X$, $h_n':Y_n\to Y$, and $f_n: X_n\to Y_n$ for each $n$.
}

\begin{prop} \blue{Let $\wt{M}$ and $\wt{N\,}\!$ be branched $\Zp$-covers with notation as above, and let $f_0:N\to M$ be a branched cover. Then}
taking a 
\bred{morphism $f:\wt{M}\to \wt{N\,}\!$} is equivalent to taking a homomorphism $\iota:\Zp\to \Zp$ which commutes with the defining homomorphism $\tau,\tau'$ and $f_{0*}:\pi_1(Y)\to \pi_1(X)$. 
\bred{Suppose $f$ is Galois. Then} 
the natural maps $G_n\to G$ are isomorphisms if and only if corresponding $\iota:\Zp\to \Zp$ is an isomorphism. 
\bred{In other words,  $f$ is equivariant Galois if and only if $f$ is $\varprojlim\Gal(h_n)=\Zp$-equivariant.} \end{prop}

\begin{proof} 
Consider the following commutative diagram with exact rows. $$ 
\xymatrix{
0\ar[r] & \pi_1(Y_n) \ar[r] \ar^{f_{n*}}[d] & \pi_1(Y) \ar^{\tau'_n}[r] \ar^{f_{\red 0*}}[d]& \Z/p^n\Z \ar[r] \ar^{\iota_n}[d] &0\\%
0\ar[r]  & \pi_1(X_n)  \ar[r] & \pi_1(X)  \ar^{\tau_n}[r] & \Z/p^n\Z \ar[r] &0}
$$ 
Then it can be seen that taking $\(f_n\)_n$ and taking $\iota:\Zp\to \Zp$ is equivalent, and that $f_{\red 0}$ is Galois if and only all the $f_n$ are Galois. 
If Galois, since 
$G_n=\pi_1(X_n)/f_{n*}(\pi_1(Y_n))$, the snake lemma yields an exact sequence $0\to \Ker(\iota_n)\to G_n \to G_{\red 0} \to \Coker(\iota_n)\to 0$. Thus the natural \blue{maps} $G_n\to G_{\red 0}$ \blue{are isomorphisms} if and only if $\iota$ is an isomorphism. \end{proof}
Even if $\iota$ is not an isomorphism, $G_n$'s can be isomorphic to each other \blue{non-canonically}.
Indeed, $\Ker \iota_n$ and $\Coker \iota_n$ are isomorphic. 
For instance, if $\iota$ is the multiplication by $p$ and $G_{\red 0}=\Z/p\Z$, then 
the situation resembles to the case of $k_1/k$ in the previous subsection. 
We can also define a branched cover of branched $\Zp$-cover which may not be Galois, however we omit here. 


Next, we see some examples. The following lemma is obvious. 

\begin{lem} Let $f:N\to M$ be a branched cover of 3-manifolds with degree $p$. Let $K$ be a knot in $M$, let $K'$ be a component of $f^{-1}(K)$, 
and let $\mu,\mu'$ denote the meridians of $K,K'$ respectively.  
\blue{If} $K$ is inert or decomposed, then $f(\mu')=\mu$.
If $K$ is branched, then $f(\mu')=p\mu$. 
\end{lem} 
This lemma ensures that the following settings give examples. 

\begin{exa} \blue{ 
Let $\mca{L}=K\sqcup \ol{S}$ be a 2-component link in $M=S^3$ 
and  let $f_{\red 0}:N\to M$ be a branched cover of degree $p$ with the branch link $\ol{S}$. 
Put $X=M-\mca{L}$, $Y=N-f_{\red 0}^{-1}(\mca{L})$ and $S=f_{\red 0}^{-1}(\ol{S})$. 
Let $\mu_K,\mu_{\ol{S}}\in H_1(X)$ \blue{denote} the meridians of $K$ and $\ol{S}$ respectively. 
Here we take $\iota={\rm id}_{\Zp}$.}  
 
\noindent   
(1) \blue{Suppose that $L$ is decomposed in $f_0$ as $f_0^{-1}(K)=K_1\sqcup \cdots \sqcup K_p$} (e.g., let $\mca{L}$ be the trivial 2-component link), and let 
(i) 
$\tau:H_1(X)\to \Zp; \mu_K\mapsto1, \mu_{\ol{S}}\mapsto 0,$ 
$\tau':H_1(Y)\to \Zp; \mu_{K_i}\mapsto 1, \mu_{S}\mapsto 0.$
(ii) $\tau: \mu_K\mapsto 1, \mu_{\ol{S}}\mapsto \blue{1},$ 
$\tau': \mu_{K_i}\mapsto 1, \mu_{S}\mapsto \blue{p}.$ 

\noindent 
(2) \blue{Suppose that $L$ is inert in $f_0$ as $f_{\red 0}^{-1}(K)=K'$} 
(e.g., let $\mca{L}$ be the Hopf link), and let 
(i) 
$\tau:\mu_K\mapsto1, \mu_{\ol{S}}\mapsto 0,$ 
$\tau':\mu_{\blue{K'}}\mapsto 1, \mu_{S}\mapsto 0.$ 
(ii) $\tau: \mu_K\mapsto 1, \mu_{\ol{S}}\mapsto \blue{1},$ 
$\tau': \mu_{\blue{K'}}\mapsto 1, \mu_{S}\mapsto \blue{p}.$\\
These above give 
\bred{equivariant Galois morphisms} of branched $\Zp$-covers of degree $p$. 

\blue{Next, we change the setting and 
see an example of non-Galois case.}\\   
\noindent 
(3) 
Let $K$ be a knot in $M=S^3$, put $X=M-K$. 
Let $f_{\red 0}:N\to M$ be a cover of degree $p$ branched over $K$,
and let $f_{\red 0}:Y\to X$ be the restriction to the exterior. 
Put $K':=f_{\red 0}^{-1}(K)$, and let $\mu, \mu'$ denote the meridians of $K, K'$ respectively. If we put 
$\tau: H_1(X)\to \Zp; \mu_K \mapsto 1$, $\tau': H_1(Y)\to \Zp; \mu_{K'}\mapsto 1$, and 
$\iota: \Zp\to \Zp; 1\mapsto p$, 
then we obtain an analogous situation of $k_1/k$ in the previous \blue{sub}section. 
\end{exa} 

\subsection{On conditions for $\mu=0$} 
\blue{In this subsection, we recall analogous natures of cyclotomic and anti-cyclotomic $\Zp$-extensions for branched $\Zp$-covers investigated in our another paper \cite{Ueki3}.} 

\blue{
In a cyclotomic $\Zp$-extension over a number field, 
any non-$p$ prime is finitely decomposed, 
and in an anti-cyclotomic $\Zp$-extension over a CM-field, 
any non-$p$ prime is totally decomposed. 
For a branched $\Zp$-covers, instead of ``any non-$p$ primes'', we focus on a certain link. 
\begin{prop}[\cite{Ueki3} Proposition 5.1] 
Let $M$ be a 3-manifold, let $L$ be a link in $M$ consisting of null-homologous components, and let $\wt{M}=\(h_n:M_n \to M\)_n$ be the TLN-$\Zp$-cover over $(M,L)$. 
Let $K$ be a knot in $M-L$. 
If $\lk(K,L) \neq 0$, then $K$ is finitely decomposed into a $p^{v_p({\rm lk}(K,L))}$-component link in $\wt{M}$. 
If $\lk(K,L)=0$, then $K$ is totally decomposed in $\wt{M}$. 
\end{prop} 
Let $f:\wt{N\,}\!\to \wt{M}$ be 
\bred{an equivariant Galois morphism} of branched $\Zp$-covers, let $\ol{S}\subset M$ denote the branch link of $f_{\red 0}:N\to M$, and suppose $\ol{S}\cap L=\phi$. 
If $\ol{S}$ is finitely decomposed in $\wt{M}$, then $f:\wt{N\,}\!\to \wt{M}$ \emph{resembles} a $p$-extension of \magenta{a} cyclotomic $\Zp$-field. 
If $\ol{S}$ is \red{totally} decomposed in $\wt{M}$, then $f:\wt{N\,}\!\to \wt{M}$ \emph{resembles} a $p$-extension of \magenta{an} anti-cyclotomic $\Zp$-field.} 

\blue{
In \cite{Ueki3}, we established an analogue of relative genus theory and studied the behaviors of Iwasawa $\mu$-invariants by following Iwasawa's argument in \cite{Iwasawa1973}. 
Suppose that $\wt{M}$ and $\wt{N\,}\!$ consist of $\Q$HS$^3$'s and that $f:\wt{N\,}\!\to \wt{M}$ is of degree $p$. 
If $\ol{S}$ is finitely decomposed in $\wt{M}$, then by \cite{Ueki3} Theorem 5.2, $\mu_{\wt{M}}=0$ if and only if $\mu_{\wt{N\,}\!}=0$. 
If 
$\ol{S}$ is infinitely decomposed in $\wt{M}$, then by \cite{Ueki3} Theorem 5.3, we have 
$\mu_{\wt{N\,}\!}\geq \#(\text{components of }\ol{S})$.
}
\blue{As a consequence, 
we have the following theorem:} 

\begin{thm}\label{mu} 
Let $f:\wt{N\,}\!\to \wt{M}$ be 
\bred{an equivariant Galois morphism} of degree $p$ of branched $\Zp$-covers of $\Q$HS$^3$'s, 
let $\ol{S}$ 
\blue{denote} the branch link of $f_0:N\to M$, and suppose $\mu_{\wt{M}}=0$. 
Then $\ol{S}$ is finitely decomposed in $\wt{M}$ if and only if $\mu_{\wt{N\,}\!}=0$. 
\end{thm}

\subsection{Tate cohomologies of branched $\Zp$-covers} 
\magenta{
In Iwasawa's second proof for Kida's formula in \cite{Iwasawa1981}, 
the following assertions were used: 
\begin{prop}[\cite{Iwasawa1981}]
Let $F_\infty/k_\infty$ be an extension of \cyan{a cyclotomic $\Zp$-field} of degree $p$, 
and let $S$ denote the set of ramified primes in $F_\infty$. Then for each $i\in \Z$,\\ 
(1) {\rm [Lemma 5]} The equality $\wh{H}^i(G, F_\infty^*)=0$ holds.\\ 
(2)\footnote{\magenta{It seems that this assertion was used implicitly.}} Let $\mca{O}_{F_\infty,S}^{*-}$ denote the minus part of $\mca{O}_{F_\infty,S}^*$ and put $\hbar^-_{i}:={\rm rank}\wh{H}^i(G,\mca{O}_{F_\infty,S}^{*-})$. 
If $k_\infty$ contains every $p$-power-th roots of unity, then $\hbar^-_2-\hbar^-_1=-1$.\\  
(3) {\rm [Lemma 1]} The equality $\wh{H}^i(G,I_{F_\infty,S})=0$ holds.\\ 
(4) The equality $q(\mca{O}_{F_\infty,S}^*)=q(\mca{O}_{F_\infty}^*)p^{\#S}$ holds.  
\end{prop}}

We consider \magenta{their analogues} in the following. 
For a branched $\Zp$-cover $\wt{M}$, we fix CW-structures or PL-structures compatible with 
the inverse system, and 
define the modules of \emph{chains, cycles, boundaries,} and \emph{homologies} by the direct limits with respect to the transfer maps. 
In addition, we also define \emph{$S$-chains} and others as in Subsection 2.3. 

Moreover, for 
\bred{an equivariant Galois morphism} of branched $\Zp$-covers $f:\wt{N\,}\!\to \wt{M}$, we put $G=\Gal(f)$, and fix CW-structures 
or PL-structures compatible with all the \bred{maps in the system}. 
Let $\ol{S}\subset M$ be a link, 
put $S=f_{\red 0}^{-1}(\ol{S})$, \blue{$S_n=h_n'^{-1}(S)$, and let $\wt{S}$ denote} \bred{the inverse system $\{S_n\}_n$.} 
Since the direct limit functor is exact, every exact sequence in Section 2
holds for $\wt{N\,}\!$ and $(\wt{N\,}\!,\wt{S})$, and 
the following proposition is obtained. 

\begin{prop}\label{ZpTate} Let the situation be as above. Then \magenta{for each $i\in \Z$},\\ 
(1) \blue{The} equation $\wh{H}^i(G,C_2(\wt{N\,}\!))=0$ holds.\\
(2) If $\wt{M}$ and $\wt{N\,}\!$ consist of $\Q$HS$^3$\blue{'s},  then $\wh{H}^i(G,Z_2(\wt{N\,}\!))\cong \wh{H}^{i+1}(G,\Z)$. 
If $G\cong \Z/p\Z$ in addition, then $\hbar_i:=\rank \wh{H}^i(G,Z_2(\wt{N\,}\!))$ satisf\blue{ies} $\hbar_1=1,\hbar_2=0$.\\
(3) If $G\cong \Z/p\Z$, then $\wh{H}^i(G,Z_1(\wt{N\,}\!)_\wt{S})=0$.\\
(4) Suppose that $G\cong \Z/p\Z$, 
\violet{$S_n$'s consist} of \blue{$\Q$}-null-homologous components, $\ol{S}$ is properly branched in $f_{\red 0}$, and $\ol{S}$ is \blue{infinitely inert (equivalent to say, unbranched and finitely decomposed)} 
in $\wt{M}$. 
Then $S$ is also \blue{infinitely inert} 
\blue{in} $\wt{N\,}\!$. 
Let $s$ denote the number of \blue{connected} component of \bred{$\mca{S}:=\varprojlim S_n$}, 
which is equal to that of $S_n$ for $n\gg 0$. 
Then there is an isomorphism $Z_2(\wt{N\,}\!)_\wt{S}/Z_2(\wt{N\,}\!)\cong \blue{B_1(\wt{N\,}\!)\cap }Z_1(\wt{S})$ \blue{for the subgroup} $\blue{B_1(\wt{N\,}\!)\cap }Z_1(\wt{S})< Z_1(\wt{S})\cong \Z^s$ 
of finite index, and satisfies 
 $\wh{H}^1(G,Z_2(\wt{N\,}\!)_\wt{S}/Z_2(\wt{N\,}\!))=0$, $\wh{H}^2(G,Z_2(\wt{N\,}\!)_\wt{S}/Z_2(\wt{N\,}\!)) \cong (\Z/p\Z)^s$. 
\end{prop} 

\begin{proof} 
\blue{Let $n\gg0$.}  
The transfer map $h'$$_{n,n+1}^!$ is an isomorphism on $H_3(N_n)\cong \Z$, 
and is the multiplication by $p$ on $H_0(N_n)\cong \Z$ and on $H_0(h_n^{-1}(S))\cong \Z^s$.
If $S$ is totally inert in $\wt{N\,}\!$, then $h'$$_{n,n+1}^!={\rm id}$ on $Z_1(h'$$_n^{-1}(S))\cong \Z^s$. 
Since ${\rm Nr}=(\text{multiplication by }p)$ is a surjection on a divisible group $\displaystyle \varinjlim_{\times p}\Z$, \magenta{the assertion} follows from Propositions 2.3, 2.5, and 2.6. 
\end{proof}


\section{Kida's formula for extensions of $\Zp$-fields}
In this section, we review the background of our main result of this paper. 
The following is a classical theorem. 
\begin{thm}[The Riemann--Hurwitz formula] Let $f:R'\to R$ be an $n$-fold covering of compact, connected Riemann surfaces and let $g$ and $g'$ denote 
the genera of $R$ and $R'$ respectively. 
the ramification indices $e(P')$ of $P'\in R'$ satisfy 
$$2g'-2=(2g-2)n+\sum_{P'\in R'}(e(P')-1).$$
\end{thm} 
 
Y.\ Kida proved a highly interesting analogue of this formula for number fields. 
\magenta{For each $k\subset \C$, we set $k_+=k\cap \R$.}  
\begin{thm}[Kida's formula, \cite{Kida1980}]
Let $p$ be an odd prime, 
let $F/k$ be a finite Galois $p$-extension of CM-fields, 
and let $F_\infty$ and $k_\infty$ denote the cyclotomic $\Zp$-extensions of $F$ and $k$ respectively. 
If $\mu_{\magenta{k_\infty/k}}=0$, then $\mu_{\magenta{F_\infty/F}}=0$, and 
$$\lambda^{-}_{\magenta{F_\infty/F}}-\delta=(F_\infty:k_\infty)(\lambda^{-}_{\magenta{k_\infty/k}}-\delta) +(\sum (e_w-1) -\sum(e_{w_+}-1)),$$
where \magenta{$\lambda^-$ denotes the $\lambda$-invariant of the minus part,}  
$\delta$ is 1 or 0 according \magenta{to whether} $k$ contains a primitive $p$-th root of unity or not, 
\magenta{$w$ (resp. $w_+$) runs through all the non-$p$ primes of $F_\infty$ (resp. $F_{\infty,+}$), and 
$e_w$ (resp. $e_{w_+}$) denotes the ramification index of $w$ (resp. $w_+$) in $F_\infty/k_\infty$ (resp. $F_{\infty,+}/k_{\infty,+}$). }  
\end{thm}

Following the method of Chevalley--Weil  \cite{CWH1934},  
K.\ Iwasawa gave an alternative proof for Kida's formula, 
with use of $p$-adic representation theory of finite groups.
He also gave a less explicit formula for a more general situation, 
from which his second proof follows: 

\begin{thm}[Iwasawa \cite{Iwasawa1981}, a corollary of Theorem 6]
Let $p$ be a prime number, 
let \blue{$k_\infty$} be a cyclotomic $\Zp$-field, 
and let \blue{$F_\infty$} be a cyclic extension of degree $p$ over $k_\infty$ unramified at every infinite place of $k_\infty$ \magenta{with $G=\Gal(k_\infty/k)$}. Assume that $\mu_{k_\infty}=0$. Then $\mu_{F_\infty}=0$, and 
$$\lambda_{F_\infty}=p\lambda_{k_\infty}+\sum_{w}(e_w-1)+(p-1)(\hbar_{2}-\hbar_{1}),$$ 
where $w$ ranges over all non-$p$ \magenta{primes} on $F_\infty$, $e_w$ denotes the ramification index of $w$ in $F_\infty/k_\infty$, and $\hbar_{i}$ denotes the $p$-rank of the abelian group $H^{i}(G, \O_{F_\infty}^{*})$ for $i=1,2$. 
\end{thm}

\magenta{An application of this formula for totally real number fields with some cohomological study of units was given in (\cite{FKOT1997}).} 

\section{Kida's formula for branched $\Zp$-covers}
\subsection{Main theorem and example} 
\blue{An analogue of the original Kida's formula (Theorem 6.2) is stated as follows:} 
\begin{thm}[Kida's formula]\label{Main}  
Let $f:\wt{N\,}\!\to \wt{M}$ be 
\bred{an equivariant Galois morphism of degree $p$-power} of 
branched $\Zp$-covers of $\Q$HS$^3$. 
Let $\ol{S}\subset M$ denote the branch link of $f_{\red 0}:N\to M$
\blue{and put $S=f_{\red 0}^{-1}(\ol{S})$.} 
\blue{If $\ol{S}$ is infinitely inert in $\wt{M}$, then so is $S$ in $\wt{N\,}\!$, and $S_n\blue{:=h_n'^{-1}(S)}$ is the branch link of $f_n$.} 
If in addition $\mu_{\wt{M}}=0$, then $\mu_{\wt{N\,}\!}=0$. 
\bred{Suppose that any component of $\ol{S}$ is not inert in $f_0$.}  
\blue{For each component $w=(w_n)_n$ of \blue{$\mca{S}:=\varprojlim_n S_n$}, let $e_w$ denote the branch index of $w$ defined as that of $w_n$ in $f_n$ for $n\gg 0$.} Then 
 $$\lambda_{\wt{N\,}\!}-1=\deg(f)(\lambda_{\wt{M}}-1)+\sum_{w\subset \blue{\mca{S}}}(e_w-1).$$
\end{thm}

\blue{By Kida's method in \cite{Kida1980}, this formula is deduced from the case of degree $p$:}  


\begin{lem}[The case of degree $p$] \label{deg=p}
Let $f:\wt{N\,}\!\to \wt{M}$ be a branched cover of degree $p$ of branched $\Zp$-covers of $\Q$HS$^3$, 
let $\ol{S}\subset M$ be a link containing the branch link of $f_{\red 0}:N\to M$, 
\blue{put $S=f_{\red 0}^{-1}(\ol{S})$, and let $\mca{S}$ denote the inverse limit of the preimages of $\ol{S}$ in $\wt{N\,}\!$.}  
If $\ol{S}$ is \blue{infinitely} 
inert in $\wt{M}$, then $S$ is \blue{infinitely} 
 inert in $\wt{N\,}\!$. 
If in addition $\mu_{\wt{M}}=0$, then $\mu_{\wt{N\,}\!}=0$. 
\bred{Suppose that any component of $\ol{S}$ is not inert in $f_0$.} 
\blue{Let} $e_w$ denote the branch index for each component $w$ of \blue{$\mca{S}$}. Then  
$$\lambda_{\wt{N\,}\!}-1=p(\lambda_{\wt{M}}-1)+\sum_{w\subset \blue{\mca{S}}}(e_w-1).$$
\end{lem}

\blue{We give their proofs in the subsequent subsections.}

\begin{rem} 
(1) As an intermediate result, we \blue{prove} 
the following: 
\blue{Let the setting be as in the Lemma above.}  
If we take CW-structures or PL-structures compatible with all the covering maps \blue{and} put $\hbar_i\blue{=}\rank \wh{H}^i(G,Z_2(\wt{N\,}\!))$ \magenta{for $G=\Gal(f)$}, 
then 
$$\lambda_{\wt{N\,}\!}=p\lambda_{\wt{M}}+\sum_{w\subset \blue{\mca{S}}}(e_w-1)+(p-1)(\hbar_2-\hbar_1).$$ 
This formula is a more direct analogue of Theorem 6.3 (a corollary of  Iwasawa \cite{Iwasawa1981}, Theorem 6). 

\noindent (2)  If we fix a $\Z$-basis of $Z_1(\wt{N\,}\!)$ containing \bred{the components of $\mca{S}$} and take the sum over it, then $w$\bred{'s} can be seen as analogues of places. 


\noindent 
\blue{(3)} The original proof by Kida (\cite{Kida1980}) was given by genus theory. 
\magenta{We expect an alternative proof of our formula with use of genus theory for 3-manifolds, which was established in \cite{Ueki3}. 
There are also analytic proofs with use of $p$-adic $L$-functions by Grass and Sinnott (\cite{Gras1978}, \cite{Sinnott1984}). It seems interesting to consider their analogues, examine analogues of $p$-adic $L$-functions, 
and explore an analogue of the Iwasawa main conjecture.}  
\end{rem} 
Here is an example: 
\begin{exa} 
Let $p=3$, let $N=M=S^3$, and let $f_{\red 0}:N\to M$ be a branched cover of degree $p$ branched over an unknot $\ol{S}$. 
Let $L=K\cup K'$ be a \magenta{Hopf} link as in the figure below, so that $\lk(K,\ol{S})=3, \lk(K',\ol{S})=\lk(K,K')=1$.
Then, $\wt{K}:=\red{f_0}^{-1}(K)=K_{1}\sqcup K_{2}\sqcup K_{3}$ is a Borromean ring, $\wt{K'}:=\red{f_0}^{-1}(K')$ is an unknot, and $\lk(K_{i},\wt{K'})=1$ for $i=1,2,3$. 

Since $L$ and $L'=f_{\red 0}^{-1}(L)$ consist of null-homologous components, 
\magenta{the} TLN-$\Zp$-cover\magenta{s} $\wt{M}$ and $\wt{N\,}\!$ are defined. 
Since $\lk(L,\ol{S})=\lk(L', S)=4 \not \equiv 0 \mod 3$, 
$\ol{S}$ and $S:=f_{\red 0}^{-1}(\ol{S})$ are totally inert in $\wt{M}$ and $\wt{N\,}\!$ respectively. 
Since the defining homomorphism $\tau,\tau'$ commute with 
$f_{\red 0*}:\pi_1(N-L'\cup S)\to \pi_1(M-L\cup \ol{S})$ and $\iota={\rm id}_{\Zp}$, a branched cover \magenta{$f:\wt{N\,}\!\to \wt{M}$} of branched $\Zp$-covers of degree $p$  is defined. 

\magenta{By Hosokawa's result (\cite{Hosokawa1958}), Theorem 1), 
the value of Hosokawa polynomials $H_{L}(t)$ and $H_{L'}(t)$ at $t=1$ are
calculated by using the linking numbers of components of $L$ and $L'$,} and satisfy $H_{L}(1)=H_{L'}(1)=\pm1$. 
\magenta{By a result of Kadokami--Mizusawa (\cite{KM2008}, Proposition 4.1), the fact $p \not | \pm1$ implies that ${\wt{M}}$ and $\wt{N\,}\!$ consist of $\Q$HS$^3$'s, and their Iwasawa invariants satisfy } 
$\lambda_{\wt{M}}=2-1=1,$ $\mu_{\wt{M}}=\nu_{\wt{M}}=0$, $\lambda_{\wt{N\,}\!}=4-1=3,$ and $\mu_{\wt{N\,}\!}=\nu_{\wt{N\,}\!}=0$. 
\magenta{Thus we have observed that Kida's formula holds for this case: 
$$3-1=3\cdot (1-1)+(3-1).$$}
\begin{center} \includegraphics[width=13cm]{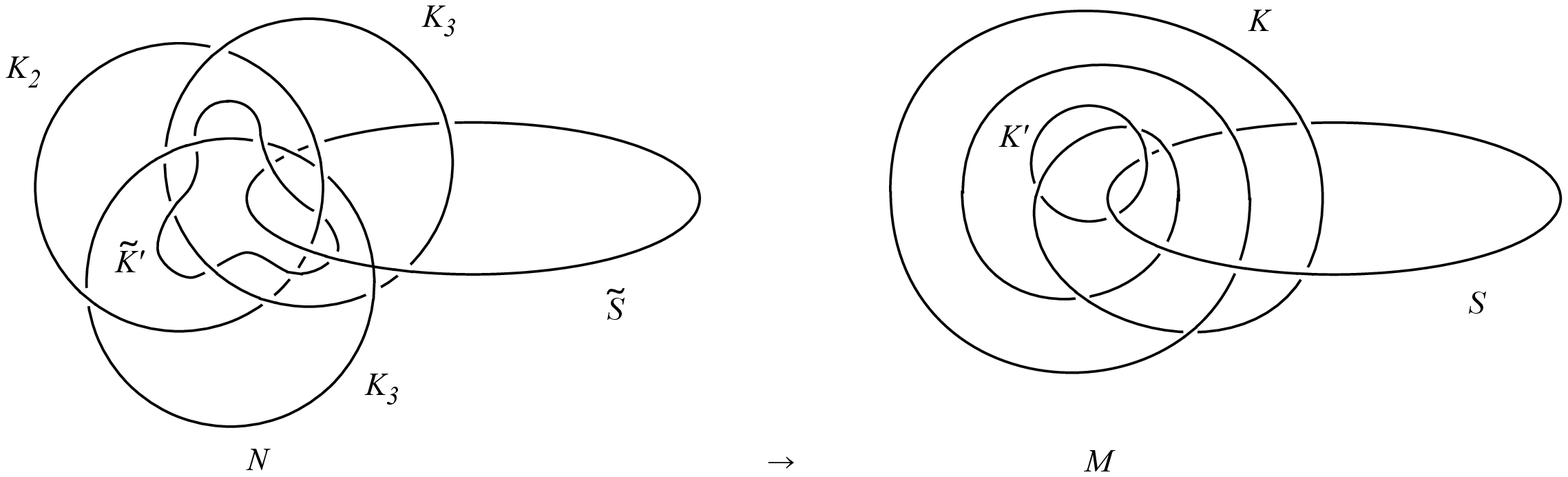}
\end{center} 

\magenta{
Let $p$ be an arbitrary prime number. In the example above, if we replace $p=3$ by $p^r$ for any $r\in \N$ and take a similar $K$ with ${\rm lk}(K,S)=p^r$, then we obtain a $\Zp$-cover $\wt{N\,}\!$ with $\lambda_{\wt{N\,}\!}=p^r$ by Kida's formula (Theorem \ref{Main}). 
Indeed, we can verify $H_{L'}(1)=\pm1$ in a similar way.}  
%
%
%
\end{exa} 
\subsection{Proof for degree $p$ (after Iwasawa)} 
In this subsection, following Iwasawa's second proof in \cite{Iwasawa1981}, we prove our formula for the case of degree $p$ (Lemma 7.2). 
%
\begin{proof}[Lemma \ref{deg=p}]
The assertion on the $\mu$-invariants is done by Theorem \ref{mu}. 
Let $\wt{S}=$ $\(h_n'^{-1}(S)\)_n$ denote the inverse system over $S$ in $\wt{N\,}\!$. 
For $\wt{S}$-chains and others, we use the notation in Subsection 5.4. 

We have $\wh{H}^n(H_1(\wt{N\,}\!)_{\wt{S}[p]})\cong \wh{H}^n(Z_2(\wt{N\,}\!)_\wt{S})$. 
Indeed, consider \blue{the} exact sequences 
$0\to B_1(\wt{N\,}\!)_\wt{S}$ $\to Z_1(\wt{N\,}\!)_\wt{S}\to H_1(\wt{N\,}\!)_\wt{S}\to 0$ and   
$0\to Z_2(\wt{N\,}\!)_\wt{S}\to C_2(\wt{M})\to B_1(\wt{N\,}\!)_\wt{S}\to0$, and note $G\cong \Z/p\Z$. Then \magenta{Proposition} \ref{ZpTate} (1) and (3) show   
$\wh{H}^n(H_1(\wt{N\,}\!)_{\wt{S}[p]})\cong \wh{H}^n(H_1(\wt{N\,}\!)_\wt{S})\congto \wh{H}^{n+1}(B_1(\wt{N\,}\!)_\wt{S}) 
\congto \wh{H}^{n+2}(Z_2(\wt{N\,}\!)_\wt{S}) \cong  \wh{H}^{n}(Z_2(\wt{N\,}\!)_\wt{S})$.

Put $A_\wt{S}:=H_1(\wt{N\,}\!)_{\wt{S}[p]}\blue{=H_1(\wt{N\,}\!)/\iota_*(H_1(\wt{\wt{S}}))}$, $A:=H_1(\wt{N\,}\!)_{[p]}$. 
Since $\mu_{\wt{N\,}\!}=0$, 
Proposition \ref{injlim} gives an isomorphism $A\cong (\Qp/\Zp)^{\lambda_{\wt{N\,}\!}}$. 
\blue{Since $\wt{N\,}$ consists of $\Q$HS$^3$, $S_n$ consists of $\Q$-null-homologous components for each $n$. If $S_n$ is inert in $N_{n+1}\to N_n$, then the restriction of the transfer $H_1(N_n)\to H_1(N_{n+1})$ to the subgroups generated by all the components of $S_n$ and $S_{n+1}$ is surjective, and hence $\iota_*(H_1(\wt{S}))$ is finite. 
Hence there is a surjective homomorphism $A\surj A_\wt{S}$ with finite kernel, and 
we have $A_\wt{S}\cong (\Qp/\Zp)^{\lambda_{\wt{N\,}\!}}$.}  


We consider a linear representation of $G=\Z/p\Z$. 
Put $X_p:=\Zp[G], Z_{p-1}:=(1-\sigma)X_p, X_1:=X_p/X_{p-1}$ $=\Zp, A_i=\Hom_{\Zp}(X_i,\Qp/\Zp)$ for $i=1,p-1,p$. 
Then there are decompositions $A=A_1^{a_1}\oplus A_{p-1}^{a_{p-1}}\oplus A_p^{a_p}$, $a_i\in \N$. 
By using a functor $V:A\mapsto \Hom(A,\Qp/\Zp)$, we put $V_i:=V(A_i), \pi_i:G\to \GL(V_i)$, $i=1,p-1,p$. 
Then $\pi_1$ is the trivial representation, $\pi_{p-1}$ is the unique faithful irreducible representation, and $\pi_p=\pi_1\oplus \pi_{p-1}=\pi_G$. 
Put $V=V(A)$. Then $G\act A$ induces a representation $\pi_{f}:G\to \GL(V)$, 
and there is a decomposition $\pi_{f}=a_1\pi_1\oplus a_{p-1}\pi_{p-1}\oplus a_p\pi_p$. 

Next, we calculate $a_i$'s. Since $G\cong \Z/p\Z$, 
the Tate cohomologies of $G\act A_i$ are 
abelian groups of exponent $p$, and 
their $p$-ranks satisfy $r(\wh{H}^1(G,A_i))=1,0,0, r(\wh{H}^2(G,A_i))=0,1,0$ for $i=1,p-1,p$. 
Put 
$r_n=r(\wh{H}^n(G,A))=r(\wh{H}^n(G,$ $Z_2(\wt{N\,}\!)_\wt{S}))$. 
Then the decomposition of $A_\wt{S}$ yields $r_1=a_1, r_2=a_{p-1}$. 

We calculate the Herbrand quotient of $\wt{S}$-2-cycles. 
(Recall that for a cyclic group $G$ and $G$-module $B$, 
$q(B)=\# \wh{H}^0(B)/\# \wh{H}^1(B)$ is called the Herbrand quotient. 
Herbrand's lemma asserts that in a short exact sequence,  
if $q(\, )$ is defined for two terms, then so for the rest term, and its Tate cohomologies are finite.)  
Now $\wh{H}^n(G,Z_2(\wt{N\,}\!)_\wt{S})$ is finite 
and $q(Z_2(\wt{N\,}\!)_\wt{S})$ is defined. 
Put $\hbar_n=r(\wh{H}^n(G,Z_2(\wt{N\,}\!)))$ 
for $Z_2(\wt{N\,}\!)<Z_2(\wt{N\,}\!)_\wt{S}$. 
Then $q(Z_2(\wt{N\,}\!)_\wt{S}/Z_2(\wt{N\,}\!))=p^s$ by
\magenta{Proposition} \ref{ZpTate} (4), 
and $q(Z_2(\wt{N\,}\!)_\wt{S})$ $=q(Z_2(\wt{N\,}\!))p^s$ by Herbrand's lemma. 
Hence $\hbar_1, \hbar_2$ are finite and $a_{p-1}-a_1=\hbar_2-\hbar_1+s$ holds.  

Since $\mu_{\wt{M}}=0$, there is an isomorphism $\bar{A}_\wt{S}:=H_1(\wt{M})_{\ol{\wt{S}}} \cong (\Qp/\Zp)^{\lambda_{\wt{M}}}$ by Proposition \ref{injlim}. 
\magenta{Proposition} \ref{ZpTate} (1) asserts $\wh{H}^2(C_2(\wt{N\,}\!))=0$, and hence $\wh{H}^1(Z_2(\wt{N\,}\!)_\wt{S})\cong (B_1(\wt{N\,}\!)_\wt{S})^G/B_1(\wt{M})_\wt{S}$, 
$\wh{H}^2(Z_2(\wt{N\,}\!)_\wt{S})\cong \Coker((Z_1(\wt{N\,}\!)_\wt{S})^G \to (H_1(\wt{N\,}\!)_\wt{S})^G)$. 
Since the first term is finite, maps $H_1(\wt{M})_{\wt{\ol{S}}}$, $\to H_1(\wt{N\,}\!)_\wt{S}$ 
and $\bar{A_\wt{S}}\to A_\wt{S}^G$ have finite kernels and cokernels, \bred{where we put $\wt{\ol{S}}=\(h_n^{-1}(\ol{S})\)_n$. }
Hence the decomposition of $A_\wt{S}$ tells $\lambda_{\wt{M}}=a_1+a_p$. 

Thereby we have obtained $\pi_{f}=a_1\pi_1\oplus a_{p-1}\pi_{p-1}\oplus a_p\pi_p
=(a_1+a_p)\pi_p \oplus (a_{p-1}-a_1)\pi_{p-1}
=\lambda_K \pi_G \oplus \blue{s}\pi_{p-1}\oplus (\hbar_2-\hbar_1)\pi_{p-1}$. 
(If $\hbar_2-\hbar_1<0$, the right hand side of the formula should be interpreted as a difference of two representations.) 
Comparing the degrees of both side, we obtain 
$$\lambda_{\wt{N\,}\!}=p\lambda_{\wt{M}}+\sum_{w\subset \bred{\mca{S}}}(e_w-1)+(p-1)(\hbar_2-\hbar_1).$$ 

By \magenta{Proposition} \ref{ZpTate} (2), \blue{we have} $\hbar_2-\hbar_1=-1$, \magenta{and obtain the desired formula.}
\end{proof}

\subsection{Proof for degree $p$-power (after Kida)} 
Following Kida's argument in \cite{Kida1980}, we deduce our formula (Theorem 7.1) from the case of degree $p$ (Lemma 7.2). 
\begin{proof}[Theorem \ref{Main}] 
Since every nontrivial finite $p$-group has nontrivial center, 
a Galois branched cover of $p$-power degree can be realized 
as a sequence of branched covers of degree $p$. 
Suppose $\deg(f)=p^n$. 
We prove the formula by induction on $m$. 
The assertion is trivial for $n=0$. 

Assume the assertion is true for $n\leq m$. Then for $n=m+1$, 
there is a subcover $\wt{N\,}\!\to \wt{H\,}\!$ of $\wt{N\,}\!\to \wt{M}$ of degree $p$. 
Let $w$ run through all the components of \blue{$\wt{S}$}, and $u$ run through the image of \blue{$\wt{S}$} in $\wt{H\,}\!$. By the hypothesis of the induction and by Lemma 7.2, they satisfy
$\lambda_{\wt{N\,}\!}-1=p(\lambda_{\wt{H\,}\!} -1)+\sum_w (e(w/u)-1)$ and  
$\lambda_{\wt{H\,}\!}-1=p^m(\lambda_{\wt{M}} -1)+\sum_u (e(u/v)-1)$. 
Here $u$ and $v$ denote the images of each $w$ in $\wt{H\,}\!$ and $\wt{M}$ respectively, $e(w/u)$ and $e(u/v)$ denote their branch indices in $\wt{N\,}\!\to \wt{H\,}\!$ and $\wt{H\,}\!\to \wt{M}$ respectively, and $e(w/v)=e_w$. 

Now we have 
$\lambda_{\wt{N\,}\!} -1=p(p^m(\lambda_{\wt{M}} -1)+\sum_u (e(u/v)-1))+\sum_w (e(w/u)-1)
=p^{m+1}(\lambda_{\wt{M}} -1)+\sum_u p(e(u/v)-1) +\sum_w (e(w/u)-1)$. 
%
If $u$ is branched in ${\wt{N\,}\!}\to {\wt{H\,}\!}$, 
then the unique $w$ over $u$ satisfies $e(w/v)=pe(u/v)$, $e(w/u)=p$, and hence 
$p(e(u/v)-1) + (e(w/u)-1)= e(w/v)-p+p-1=e(w/v)-1$. 
%
If $u$ is decomposed in ${\wt{N\,}\!}\to {\wt{H\,}\!}$, then 
the correction $w_1,...,w_p$ of components over $u$ satisfy $e(w_i/v)=e(u/v)$, $e(w_i/u)=1$, and hence 
$p(e(u/v)-1) +\sum_i (e(w_i/u)-1)=\sum_i (e(w_i/v)-1)$. 

Thus, due to the assumption that $S$ is not inert in $f_{\red 0}:N\to M$, we have generalized the formula to the case of $p$-power degree. 

(If $u$ is inert in ${\wt{N\,}\!}\to{\wt{H\,}\!}$, then $w$ over $u$ satisfies 
$e(w/v)=e(u/v)$, $e(w/u)=1$ and hence $p(e(u/v)-1)+ (e(w/u)-1)=p e(w/u)-p+1-1=p(e(w/u)-1)$.)
\end{proof}

\section*{Acknowledgments} 
The author would like to appreciate Yasushi Mizusawa for informing a paper on Kida's formula \magenta{(\cite{FKOT1997})}, 
Teruhisa Kadokami for discussion on significance of the notion of a branched $\Zp$-cover,  
Mikio Furuta for instructive advise from topological point of view, 
Masanori Morishita for strong \magenta{encouragement} and pointing out errors, 
and the anonymous referees for accurate comments. 
He is also very grateful to \blue{Ryo Furukawa}, Yu-ich Hirano, Tomoki Mihara, \magenta{Takayuki Morisawa, Hirofumi Niibo, Takayuki Okuda}, and Tatsuro Shimizu for fruitful discussions. 
The author is partially supported by Grant-in-Aid for JSPS Fellows (25-2241).

\bibliographystyle{amsalpha}
\bibliography{ju.Kida.ams}

\ \\ 
Jun Ueki \\
Graduate School of Mathematical Sciences, The University of Tokyo\\
3-8-1, Komaba, Meguro-ku, Tokyo, 153-8914, Japan\\ 
E-mail: \url{uekijun46@gmail.com} 

\end{document}